\documentclass[11pt,draftcls,onecolumn]{IEEEtran}

\usepackage[sort]{cite}
\usepackage{psfrag}
\usepackage{epsfig}
\usepackage{amsmath,bbm,times,stmaryrd}

\usepackage{amsthm}

 \usepackage[mathscr]{eucal}

\usepackage{tabularx}
\usepackage{multirow}
\usepackage{enumerate}

\usepackage[normal]{subfigure}

\usepackage{colortbl}
\usepackage{dcolumn}
\newcolumntype{.}{D{.}{.}{1.3}}

\usepackage{txfonts}
\usepackage[subnum]{cases}

\usepackage{equationarray}

\newcommand\dashrule{\leavevmode\xleaders\hbox{-}\hfill\kern0pt}

\def\diag{\operatorname{diag}}

\newcommand{\bc}{\bvec{c}}
\newcommand{\bd}{\bvec{d}}

\newcommand{\bg}{\bvec{g}}

\newcommand{\bs}{\bvec{s}}

\newcommand{\bu}{\bvec{u}}
\newcommand{\bv}{\bvec{v}}

\newcommand{\bx}{\bvec{x}}
\newcommand{\by}{\bvec{y}}
\newcommand{\bz}{\bvec{z}}

\newcommand{\bA}{{\bf A}}
\newcommand{\bB}{{\bf B}}

\newcommand{\bD}{{\bf D}}

\newcommand{\bF}{{\bf F}}
\newcommand{\bG}{{\bf G}}
\newcommand{\bH}{{\bf H}}
\newcommand{\bI}{{\bf I}}

\newcommand{\bK}{{\bf K}}

\newcommand{\bP}{{\bf P}}
\newcommand{\bQ}{{\bf Q}}

\newcommand{\bT}{{\bf T}}
\newcommand{\bU}{{\bf U}}
\newcommand{\bV}{{\bf V}}

\newcommand{\bX}{{\bf X}}
\newcommand{\bY}{{\bf Y}}
\newcommand{\bZ}{{\bf Z}}

\newcommand{\calI}{{\mathcal{I}}}

\newcommand{\calL}{{\mathcal{L}}}

\newcommand{\calO}{{\mathcal{O}}}

\newcommand{\bbeta}{\mbox{\boldmath $\beta$}}

\newcommand{\bdelta}{\mbox{\boldmath $\delta$}}

\newcommand{\boldeta}{\mbox{\boldmath $\eta$}}
\newcommand{\btheta}{\mbox{\boldmath $\theta$}}

\newcommand{\bsigma}{\mbox{\boldmath $\sigma$}}

\newcommand{\bGamma}{\mbox{\boldmath $\Gamma$}}

\newcommand{\bPi}{\mbox{\boldmath $\Pi$}}
\newcommand{\bSigma}{\mbox{\boldmath $\Sigma$}}

\newcommand{\bPhi}{\mbox{\boldmath $\Phi$}}
\newcommand{\bPsi}{\mbox{\boldmath $\Psi$}}

\newcommand{\1}{\mbox{\boldmath $1$}}
\newcommand{\0}{\mbox{\boldmath $0$}}

\newcommand{\be}{\begin{eqnarray}}
\newcommand{\ee}{\end{eqnarray}}

\newcommand{\matrixb}{\left[ \begin{array}}
\newcommand{\matrixe}{\end{array} \right]}

\newcommand{\tr}{\mathop{\rm tr}\nolimits}

\def\*{\circledast}

\newtheorem{lemma}{Lemma}

\newcommand{\bvec}[1]{\boldsymbol{#1}}

\newcommand{\ve}{\bvec{e}}
\newcommand{\vf}{\bvec{f}}

\def\vectorize{\operatorname{vec}}
\newcommand{\vtr}[1]{\vectorize\hspace{-.3ex}\left(#1\right)}

\newcommand{\tensor}[1]{\boldsymbol{\mathscr{\MakeUppercase{#1}}}} 

\newcommand{\tI}{\tensor{I}}

\newcommand{\tX}{\tensor{X}}
\newcommand{\tY}{\tensor{Y}}

\usepackage[vlined,ruled,commentsnumbered]{algorithm2e}

\usepackage{aurical}

\def\dvec{\operatorname{dvec}}
\def\blkdiag{\operatorname{blkdiag}}
\def\bigcircledast{\mathop{\mbox{\fontsize{18}{19}\selectfont $\circledast$}}}

\renewcommand{\bigodot}{\mathop{\mbox{\fontsize{18}{19}\selectfont$\odot$}}}

\usepackage{arydshln}

\usepackage{amssymb}
\usepackage{amsmath}
\usepackage{amsfonts}
\usepackage{mathdots}

\usepackage{url}

\usepackage{mathtools}

\usepackage{graphicx}
\graphicspath{{./Linregbound/}}

\definecolor{razzledazzlerose}{rgb}{1.0, 0.2, 0.8}
\definecolor{amethyst}{rgb}{0.6, 0.4, 0.8}
\definecolor{chromeyellow}{rgb}{1.0, 0.65, 0.0}

\renewcommand{\CommentSty}[1]{\fontsize{8.7}{9.8}\selectfont\textnormal{\texttt{#1}}\unskip}
\SetKwComment{mtcc}{\% }{}
\SetKwSwitch{Switch}{Case}{Other}{switch}{}{case}{otherwise}{end switch}%

\newcounter{example} 

\newenvironment{example}
{\refstepcounter{example}\vspace{10pt}\par\noindent 
\textbf{Example \theexample\ }
}
{}%

\makeatletter
\newsavebox{\@brx}
\newcommand{\llangle}[1][]{\savebox{\@brx}{\(\m@th{#1\langle}\)}%
  \mathopen{\copy\@brx\kern-0.5\wd\@brx\usebox{\@brx}}}
\newcommand{\rrangle}[1][]{\savebox{\@brx}{\(\m@th{#1\rangle}\)}%
  \mathclose{\copy\@brx\kern-0.5\wd\@brx\usebox{\@brx}}}
\makeatother

\def\comment#1{}

\title{Error Preserving Correction for CPD and Bounded-Norm CPD}
\author{Anh-Huy Phan, Petr Tichavsk\'{y} and Andrzej Cichocki
\thanks{A.-H. Phan and A. Cichocki are with the Lab for Advanced Brain Signal Processing, Brain Science Institute, RIKEN, Wakoshi, Japan, e-mail: (phan,cia)@brain.riken.jp.}
\thanks{A. Cichocki is also with Systems Research Institute PAS, Warsaw, Poland, and Skolkovo Institute of Science and Technology (Skoltech), Russia}
\thanks{P.  Tichavsk{\'y} is with Institute of Information Theory and Automation, Prague, Czech Republic, email: tichavsk@utia.cas.cz.}
\thanks{The work of P. Tichavsk{\'y} was supported by the Czech Science Foundation through project No. 17--00902S.}
}

\begin{document}
\maketitle

\begin{abstract}
In CANDECOMP/PARAFAC tensor decomposition, degeneracy often occurs in some difficult scenarios, 
e.g., when the rank exceeds the tensor dimension, or when the loading components are highly collinear in   several or all modes, or when CPD does not have an optimal solution. In such the cases, norms of some rank-1 terms become significantly large, and cancel each other. This makes algorithms getting stuck in local minima, while running a huge number of iterations does not improve the decomposition.
In this paper, we propose an error preservation correction method to deal with such problem. Our aim is to seek a new tensor whose norms of rank-1 tensor components are minimised in an optimization problem, while it preserves the approximation error. An alternating correction algorithm and an all-at-one algorithm have been developed for the problem. In addition, we propose a novel CPD with a bound constraint on the norm of the rank-one tensors. The method can be useful for decomposing tensors that cannot be analyzed by traditional algorithms, such as tensors corresponding to the matrix multiplication.
\end{abstract}

\section{Introduction}

In this paper, we consider the CANDECOMP/PARAFAC tensor decomposition, which approximates a tensor $\tY$ by a sum of rank-1 tensors in the form of
\be
	\tY \approx \hat{\tY} =  \sum_{r = 1}^{R} \eta_r \, \bu^{(1)}_r \circ \bu^{(2)}_r \circ \cdots \circ \bu^{(N)}_r
\ee
where $\bU^{(n)} = [\bu^{(n)}_1, \ldots, \bu^{(n)}_R]$ are factor matrices of size $I_n \times R$. 
The tensor $\tY$ is of size $I_1 \times I_2 \times \cdots \times I_N$, and its approximation is the tensor $\hat{\tY}$ of rank-$R$.
This decomposition has found numerous applications in 
identification of independent components, signals retrieval in CDMA telecommunications, extraction of hidden components from neural data, image completion and various tracking scenarios\cite{cichocki2016tensor}.


When the loading components $\bu^{(n)}_r$ are assumed to be unit-length vectors, the weight $\eta_r$ represents the Frobenius norm of the $r$-th rank-one tensor
\be
\|\eta_r \, \bu^{(1)}_r \circ \bu^{(2)}_r \circ \cdots \circ \bu^{(N)}_r\|_F^2 = \eta_r^2 \,  \|\bu^{(1)}_r\|^2  \|\bu^{(2)}_r\|^2 \cdots  \|\bu^{(N)}_r\|^2 = \eta_r^2. \notag
\ee

In some difficult decomposition scenarios, the norms of some rank-1 terms become significantly large and cancel each other. 
This is often observed when the rank exceeds the tensor dimension, or when the loading components are highly collinear in   several or all modes (swamps) \cite{MITCHELL94}. 
Moreover, it may happen that the CP does not have an optimal solution \cite{5352186,Silva2008}, because the tensor can be arbitrarily well approximated by tensors of lower rank.
%
%
This degeneracy phenomenon is reported in the literature, e.g., in\cite{Lundy,Harshman_book,MITCHELL94,Paatero97,Paattero00,CEM:CEM1236,Harshman04,journals/siammax/Stegeman12,Kolda08,cichocki2016tensor}.
Some efforts have been made to improve stability and convergence for such the cases \cite{Rayens,Paattero00}. For example, the factor loadings can be imposed additional constraints, e.g., orthogonality \cite{Rayens,5352186}, positivity or nonnegativity \cite{Paatero97,Lim-Comon}. An alternative method is to decompose the data with a regularisation to stabilise the algorithm, e.g.,
\be
	\min \quad \|\tY - \hat{\tY}\|_F^2 + \frac{\mu}{2} \, \sum_{n} \|\bU^{(n)}\|_F^2 \,. \label{eq_cpd_smooth}
\ee
The Levenberg-Marquard method solves the above problem efficiently with a relatively low computational cost when it exploits the Khatri-Rao structures of rank-one tensors\cite{Phan_fLM}.
The damping parameter $\mu$ is adaptively adjusted, namely, it should converge to zero.

In some applications, an exact CP representation is sought. An example is decomposition of tensors corresponding to the matrix multiplication. This is one of the main challenging tasks
of theory of complexity to find a minimum number of scalar multiplications required to compute a product of two matrices of given sizes. In \cite{journals/corr/TichavskyPC16}
the first term in (\ref{eq_cpd_smooth}) is minimised while the second term is constrained to a constant. \be
\min \quad \| \tY- \hat{\tY}\|_F^2 \quad \text{s.t.} \quad \sum \|\bU_n\|_F^2 \le c \notag \,.
\ee
In this way, it is possible to find an exact decomposition of the matrix multiplication tensors for certain matrix sizes and avoid convergence to singular ``solutions" where the norm of some factor matrices converges to infinity.

In this paper, we propose a novel method to deal with this challenging problem. 
Different from the existing algorithms for this kind of tensor decomposition, our aim is to correct the rank-1 tensors if  their norm is observed to be relatively high during the tensor approximation process. 
More specifically, we seek a new tensor, $\hat{\tY}$, whose norms of rank-1 tensor components are minimal, while it is still able to explain $\tY$ at the current level of approximation error. 
Continuing the decomposition with a new tensor with a lower norm will prevent CP algorithms from degeneracy and thereby improve their convergence. 
This can be achieved by solving the following constrained CP tensor approximation 
\be
\min \quad &  f(\btheta) = \|\boldeta\|_2^2 = \sum_{r = 1}^{R} \eta_r^2 \label{eq_cp_boundnorm}\\
\text{subject to} \quad &  c(\btheta) =  \|\tY - \hat{\tY}\|_F^2 \le \delta^2 \notag ,
\ee
where $\btheta$ represents a vector of all model parameters.
We call this the Error Preserving Correction (EPC) method. 
%

In Section~\ref{sec:AU}, we derive algorithms for the above constrained nonlinear optimisation: an alternating EPC algorithm and another one based on the Sequential Quadratic Programming (SQP) method to update all the parameters at a time. In the alternating algorithm, we reformulate the optimisation in (\ref{eq_cp_boundnorm}) as linear regression sub-problems with a bound constraint for the factor matrices, which in turn can be solved in closed-form through the Spherical Constrained Quadratic Programming (SCQP). 
For the SQP algorithm, we derive fast inverse of the Hessian matrix.

In the second part of the paper, together with the EPC for CPD, we propose a novel CPD with a bound constraint on the norm of rank-1 tensors
\be
\min \quad &\| \tY - \hat{\tY} \|_F^2 \quad 
\text{s.t.} \quad & \|\boldeta \|_2^2 \le \epsilon^2  \, .\label{eq_cp_bound_lda_0}
\ee
Note that the optimization problem (\ref{eq_cp_bound_lda_0}) is dual to the problem in (\ref{eq_cp_boundnorm}) and vice versa.
This method is similar but not identical to the method of \cite{journals/corr/TichavskyPC16} with a bound on the sum of squared Frobenius norm of the factor matrices.
A novel ALS algorithm and an SQP algorithm are then derived for the bounded norm CPD.

We also present a relation between the alternating EPC correction algorithm and the ordinary ALS algorithm, and a relation between the new ALS for CPD with a bound constraint and the ALS for CPD with the Tikhonov regularization given in (\ref{eq_cpd_smooth}). 

In the Simulation section, Section~\ref{sec::simulation}, we present examples of utilisation of the proposed algorithms and methods in decomposing artificially constructed tensors, tensors corresponding to the matrix multiplication, and tensor of real-world TV rating data\cite{Lundy}.

\section{Error Preserving Correction Algorithms}\label{sec:AU}

We note that the constraint function in the optimisation (\ref{eq_cp_boundnorm})  is nonlinear with respect to all the factor matrices, but linear in parameters in one factor matrix, 
or parameters in non-overlapping partitions of different factor matrices \cite{Petr_PALS,Phan_PHALS}. A simple approach to handle this kind of constrained nonlinear optimisation is to rewrite the objective function and especially the constraint function in a linear form. This can be achieved using the alternating update scheme or the Sequential Quadratic Programming method\cite{boggs_tolle_1995,Fletcher_book}.

\subsection{The alternating correction method}

In this section, we present an application of the linear regression in Appendix~\ref{sec::linreg_bounderror} in (\ref{eq_linreg_bound}) in decomposition of a tensor.

 At each iteration, we seek a new estimate of the factor matrix $\bU^{(n)}$ which reduces the objective function, while preserving the approximation error. Observing that by absorbing $\boldeta$ into the factor matrix $\bU^{(n)}$ to give ${\bU}_{\eta}^{(n)}= \bU^{(n)}  \diag(\boldeta)$, while keeping the other factor matrices $\bU^{(k)}$ fixed, $k \neq n$, the objective function becomes 
\be
	\| \boldeta \|^2 =  \| \bU^{(n)}  \diag(\boldeta) \|_F^2 =  \|{\bU}_{\eta}^{(n)} \|_F^2 \, .
\ee

The constraint is rewritten for the factor matrix $\bU_{\eta}^{(n)}$ as
\be
\|\tY - \hat{\tY}\|_F^2 &=& \|\bY_{(n)} - \bU_{\eta}^{(n)}  \bT_n^T \|_F^2  \notag \\
&=&  \|\tY\|_F^2 +  \tr(\bU_{\eta}^{(n)}  \bT_n^T \bT_n \bU_{\eta}^{(n) T}) - 2 \tr(\bY_{(n)} \, \bT_n \, \bU_{\eta}^{(n) T}) \notag \\
&=& \tr(\bU_{\eta}^{(n)} \, \bGamma_{-n}  \, \bU_{\eta}^{(n) T}) - 2 \tr(\bG_n \, \bU_{\eta}^{(n) T}) + \|\tY\|_F^2    \notag \\
&=& \| \bG_n  \bV_n \bSigma^{\frac{-1}{2}}   - \bU_{\eta}^{(n)} \bV_n \bSigma^{\frac{1}{2}} \|_F^2 + \|\tY\|_F^2 - \|\bG_n \bV_n \bSigma^{\frac{-1}{2}}\|_F^2 \label{eq_expansion_of_error}
\ee
where $\bY_{(n)}$ is the mode-$n$ matricization of $\tY$, 
 $\bT_n = \bigodot_{k \neq n} \bU^{(k)}$ is the Khatri-Rao product of all but one factor matrices, $\bG_n = \bY_{(n)} \, \bT_n$ is of size $I_n \times R$, and $\bGamma_{-n} = \bT_n^T \bT_n = \bigcircledast_{k \neq n} (\bU^{(k) T } \bU^{(k)})$ is of size $R \times R$.
 
The matrix $\bGamma_{-n}$ is assumed to be positive definite, and its EVD is denoted by $\bGamma_{-n} = \bV_n \bSigma \bV_n^T$, where $\bV_n$ is an orthonormal matrix of eigenvectors, and $\bSigma = \diag(\sigma_1 \ge \ldots \ge \sigma_R > 0)$ is a diagonal matrix of positive eigenvalues. 
Note that the matrix $\bV_n$ comprises right singular vectors associated with the singular values  $\sigma_r^{\frac{1}{2}}$ of $\bT_n$. 

Let $\bF_n = \bG_n \bV_n$.
 The optimisation problem (\ref{eq_cp_boundnorm}) becomes the linear regression with the bounded error constraint 
\be
\min \quad &   \|\bU_{\eta}^{(n)}\|_F^2 \label{eq_cp_boundnorm_un}\\
\text{subject to} \quad & \| \bF_n \bSigma^{\frac{-1}{2}}   - \bU_{\eta}^{(n)} \bV_n \bSigma^{\frac{1}{2}} \|_F^2  \le \delta_n^2 \notag 
\ee
where $\delta_n^2 = \delta^2 + \|\bF_n \bSigma^{\frac{-1}{2}}\|_F^2 - \|\tY\|_F^2$. 
According to Lemma~\ref{lem_linreg_bounderror} in Appendix~\ref{sec::linreg_bounderror}, the inequality constraint can be replaced by an equality constraint, and the problem can be solved in closed form by replacing $\bU_{\eta}^{(n)}$ by its vectorizaton and formulating it as a Spherical Constrained QP  (SCQP) in (\ref{eq_linreg_boundz}) for $I_n R$ parameters. An alternative method is that we apply the conversion for matrix variate in Appendix~(\ref{sec:sqp_matrixvariate}), and formulate an SCQP for only $R$ parameters. 
To this end, we perform a reparameterization
\be
\bZ_n &=& \frac{1}{\delta_n} \, (\bF_n \bSigma^{\frac{-1}{2}}   - \bU_{\eta}^{(n)} \bV_n \bSigma^{\frac{1}{2}})    , \\
\bU_{\eta}^{(n)} &=& (\bF_n \bSigma^{\frac{-1}{2}}  - \delta_n \, \bZ_n) \, 
\bSigma^{\frac{-1}{2}} \, \bV_n^T     , \label{eq_Ulda}
\ee
and represent the Frobenius norm of $\bU_{\eta}^{(n)}$ as
\be
	\|\bU_{\eta}^{(n)}\|_F^2 =  \|\bF_n \bSigma^{-1}\|_F^2 + \delta_n^2 \tr(\bZ_n \, \bSigma^{-1} \, \bZ_n^T) - 2  \delta_n \tr(\bF_n \bSigma^{\frac{-3}{2}} \bZ_n^T) \notag \,.
\ee
The matrix $\bZ_n$ of size $I_n \times R$  is a minimiser to an SCQP for matrix-variate 
\be
\min \quad &    \delta_n  \tr(\bZ_n \, \bSigma^{-1} \, \bZ_n^T) - 2 \tr(\bF_n \bSigma^{\frac{-3}{2}} \bZ_n^T) \label{eq_Z_SCQP}\\ 
\text{subject to} \quad & \| \bZ_n \|_F^2  = 1 \notag .
\ee

According to Lemma~\ref{lem_sqp_simplify} in Appendix~\ref{sec:scqp_identical_eig} and 
the SCQP for matrix variate in Appendix~\ref{sec:sqp_matrixvariate}, the minimiser $\bZ_n^{\star}$ can be derived 
from the minimiser $\bz^{\star} = [z_1^{\star}, \ldots, z_R^{\star}]^T$  to an SCQP of a smaller scale
\be
\min \quad &    \delta_n  \bz^T \, \bSigma^{-1} \, \bz^T - 2  \bc^T \bz \quad 
\text{s.t.} \quad    \bz^T \bz = 1 \label{eq_qps_for_z} 
\ee
where $\bc = [\ldots, \sigma_r^{\frac{-3}{2}}  \|\vf_r^{(n)}\|, \ldots]$.
For a non zero $c_r$, the $r$-th column of $\bZ_n^{\star}$ is the $r$-th column of $\bF_n$ scaled by a factor of $\frac{z_r^{\star}}{\|\vf_r^{(n)}\|}$
\be
\bz_r^{(n) \star} =  \frac{z_r^{\star}}{\|\vf_r^{(n)}\|} \, \vf^{(n)}_r \,.
\ee
Otherwise, for a zero $c_r = 0$, $\bz_r^{(n) \star}$ can be any vector of length $\|\bz_r^{(n) \star}\|^2 = ({z_r^{\star}})^{2}$.
It can also be shown that if $c_r = 0$ for $r > 1$,  then $z_r^{\star} = 0$ \cite{Phan_QPS}, hence $\bz_r^{(n)\star}$ is a zero vector.
Replacing $\bZ_n$ in (\ref{eq_Ulda}) by $\bZ_n^{\star}$ yields a new update of $\bU_{\eta}^{(n)}$.
 
 At each iteration, we update $\bU_{\eta}^{(n)}$ by a new matrix having a smaller Frobenius norm, while still preserving the approximation error $\|\tY - \hat{\tY}\|_F^2 = \delta^2$. The new estimates of $\eta_r$ and $\bu_r^{(n)}$ are respectively the $\ell_2$-norm of the vector $\bu_{\eta,r}^{(n)}$ and its $\ell_2$-normalised version
\be
\eta_r = \|\bu_{\eta,r}^{(n)}\|\, , \quad 	\bu_r^{(n)} = \frac{\bu_{\eta,r}^{(n)}}{\eta_r}   \,.
\ee

Similarly, in the next iteration, we update $\bU_{\eta}^{(n+1)}$, then normalise it to obtain the new estimate of $\bU^{(n+1)}$ and $\boldeta$. The algorithm sequentially updates all $\bU^{(n)}$, and stops when there is not any significant improvement in $\boldeta$. The Alternating Correction for Error Preservation (ACEP) is summarized in Algorithm~\ref{alg_BALS}. 
As in the ordinary ALS algorithm, the most expensive step in ACEP is the computation of $N$ matrices $\bG_n$. 
However, these terms are indeed not computed explicitly as in Step~2, but through a progressive computation for the fast computation of CP gradients \cite{Phan_fastALS}, which costs $\calO(2 R I^N)$.

 \setlength{\algomargin}{1em}
\begin{algorithm}[ht!]
\SetFillComment
\SetSideCommentRight
\CommentSty{\footnotesize}
\caption{{\tt{Alternating Error Preserving Correction for CPD (ACEP)}}\label{alg_BALS}}
\DontPrintSemicolon \SetFillComment \SetSideCommentRight
\KwIn{Data tensor $\tY$:  $(I_1 \times I_2 \times \cdots \times I_N)$,  and a rank $R$ and error bound $\delta$} 
\KwOut{$\tX = \llbracket \boldeta; \bU^{(1)}, \bU^{(2)}, \ldots, \bU^{(N)}\rrbracket $ of rank $R$ such that  
$\min \|\boldeta\|_2^2$ s.t. $\|\tY - \tX \|_F^2  \le \delta^2 $}
\Begin{
\nl Initialize $\tX = \llbracket \boldeta; \bU^{(1)}, \bU^{(2)}, \ldots, \bU^{(N)}\rrbracket$ such that $\|\tY - \tX \|_F^2  \le \delta^2 $\;
\Repeat{a stopping criterion is met}{
\For{$n = 1, 2, \ldots, N$}{
	{
		\nl Compute $\bG_n = \bY_{(n)} \, \left(\bigodot_{k \neq n} \bU^{(k)} \right)$\;
		\nl Compute EVD of $\bGamma_{-n} = \bigcircledast_{k \neq n} (\bU^{(k) T } \bU^{(k)}) = \bV_n \bSigma \bV_n^T$ and  $\bF_n = \bG_n \bV_n$\;
 		\nl Solve an SCQP: \quad  $\min \quad  \delta_n  \bz^T \, \bSigma^{-1} \, \bz^T - 2  \bc^T \bz $ \quad  s.t. $\bz^T \bz = 1$ \;   
		\nl  where  $\bc = [\ldots, \sigma_r^{\frac{-3}{2}}  \|\vf_r^{(n)}\|, \ldots]^T$, $\delta_n^2 = \delta^2 + \|\bF_n \bSigma^{\frac{-1}{2}}\|_F^2 - \|\tY\|_F^2$\;
		\nl $\bZ_n  = [\ldots,  \frac{z_r}{\|\vf_r^{(n)}\|} \vf_r^{(n)}, \ldots]$\;
		\nl $\bU_{\eta}^{(n)}  =  (\bF_n \bSigma^{\frac{-1}{2}}  - \delta_n \, \bZ_n) \, 
\bSigma^{\frac{-1}{2}} \, \bV_n^T $\tcc*{see a more compact form in (\ref{eq_Ulda_compact})}
		\nl Update $\boldeta$ and  $\bU^{(n)}$:  $\eta_r = \|\bu_{\eta,r}^{(n)}\|$, 
		$\bu_r^{(n)} = \frac{\bu_{\eta,r}^{(n)}}{\eta_r} $
	}
}
}
}
\end{algorithm}

\subsection{Relation between ACEP and ALS}

Consider the case when the first column of $\bF_n$ is non-zero, i.e., $c_1 \neq 0$, hence,  $z_1^{\star} \neq 0$, and the matrix $\bZ_n^{\star}$ can be expressed as 
\be
\bZ_n^{\star} = \bar{\bF}_n \,  \diag([\ldots, z_r^{\star}, \ldots]) \, \label{eq_zrmap}
\ee
where columns of $\bar{\bF}_n$ are $\frac{\vf^{(n)}_r}{\|\vf_r^{(n)}\|}$ for non zero columns $\vf^{(n)}_r$, and zero vectors elsewhere. 

Since $\vf_1^{(n)}$ is non-zero, the minimiser $\bz^{\star}$ to the SCQP in (\ref{eq_qps_for_z}) is given in closed-form as 
\be
z_r^{\star} 
= \frac{c_r}{\|\bc\| (s_r-\lambda) }
= \frac{\|\vf_r{(n)}\| \sigma_r^{\frac{-3}{2}} }{\|\bc\| (s_r-\lambda) }
\label{eq_zr} 
\ee
where $s_r= 1 +   \, \frac{\delta_n}{\|\bc\|}(\sigma_r^{-1} - \sigma_1^{-1})$, and $\lambda$ is a unique root in $[0,1)$ of a secular function $\sum_{r}\, ({z_r^{\star}})^2 = 1$ \cite{GANDER1989815,Phan_QPS}. 

From (\ref{eq_Ulda}), (\ref{eq_zrmap}), (\ref{eq_zr}) and the definition of $\bc$, the new update of $\bU_{\eta}^{(n)}$ can be expressed in a compact form as 
\be
\bU_{\eta}^{(n)} &=& (\bF_n \bSigma^{\frac{-1}{2}}  -  \delta_n \, \bar{\bF}_n \diag([\ldots, z_r^{\star}, \ldots])) \, 
\bSigma^{\frac{-1}{2}} \, \bV_n^T \notag \\
&=&  \bF_n \diag\left(1 - \frac{\delta_n}{\|\bc\|\sqrt{\sigma_r} (s_r - \lambda)}\right) \bSigma^{-1} \bV_n^T   \, .  \label{eq_Ulda_compact}
\ee
Observed that only when $\delta_n = 0$, the above update (\ref{eq_Ulda_compact}) boils down to the ALS update 
\be
	\bU_{\eta}^{(n)}  = \bY_{(n)} \bT_n \bGamma_{-n}^{-1}. \notag 
\ee
%

%

\subsection{Sequential quadratic programming for EPC}

Similar to the ordinary ALS algorithm for the CP decomposition, the ACEP algorithm updates one factor matrix per iteration; hence, it may require many iterations to converge.
It might be useful to consider an ``all-at-once'' algorithm for the EPC, which would be analog to the nonlinear algorithms for CPD and can be combined with them to improve stability and performance of the whole computation.
The algorithm follows the idea of the sequential quadratic programming \cite{boggs_tolle_1995,Fletcher_book}.
The objective function which represents sum of Frobenius norms of rank-1 tensors is rewritten as 
\be
f(\btheta) &=& \sum_{r = 1}^R \|\bu^{(1)}_r \circ \bu^{(2)}_r \circ \cdots \circ \bu^{(N)}_r\|_F^2 \notag \\
&=& \sum_{r = 1}^R   \prod_{n = 1}^{N}  \, (\bu^{(n) T}_r \, \bu^{(n)}_r)    \, , \label{eq_cp_boundnorm2}
\ee
and the optimisation problem in (\ref{eq_cp_boundnorm}) is stated as 
\be
\min \quad f(\btheta)  \quad \text{s.t.} \quad c(\btheta) =  \|\tY - \hat{\tY}(\btheta)\|_2^2 \le \delta^2  \, \label{eq_cp_boundnorm2b}
\ee
where $\btheta = [\vtr{\bU^{(1)}}^T, \vtr{\bU^{(2)}}^T, \ldots, \vtr{\bU^{(N)}}^T]^T$.

As the derivation of the ACEP algorithm, we alternatively minimise an equivalent problem with an equality constraint
\be
\min \quad f(\btheta)  \quad \text{s.t.} \quad c(\btheta) =  \|\tY - \hat{\tY}(\btheta)\|_2^2 = \delta^2  \, .\label{eq_cp_boundnorm2_equ}
\ee 
In order to achieve this, we first construct the Lagrangian function
\be
\calL(\btheta, \lambda)  = f(\btheta) + \lambda \,  (c(\btheta) - \delta^2 ), \label{eq_augLagragian_boundnorm1}
\ee
then approximate $\calL(\btheta^{(k)} + \bd_{\theta}, \lambda^{(k)}  +  d_{\lambda})$ by a second order Taylor expansion around $(\btheta^{(k)}, \lambda^{(k)})$ 
\be
\calL(\btheta^{(k)} + \bd_{\theta}, \lambda^{(k)}  + d_{\lambda}) &\approx & \calL(\btheta^{(k)}, \lambda^{(k)}) + (\nabla \calL(\btheta^{(k)}, \lambda^{(k)}))^T   \, \bd    +  
\frac{1}{2} \,  \bd^T \, [ \nabla^{2} \calL(\btheta^{(k)}, \lambda^{(k)}) ]  \, \bd  \, , \notag 
\ee
where $\bd = \left[\bd_{\theta}^T, d_{\lambda} \right]^T$ represents the vector of increment.
This gives an approximation to the gradient $\nabla \calL(\btheta^{(k)} + \bd_{\theta}, \lambda^{(k)}  + d_{\lambda})$
\be
\nabla{\calL}(\btheta^{(k)} + \bd_{\theta}, \lambda^{(k)}  + d_{\lambda}) \approx  \nabla \calL(\btheta^{(k)}, \lambda^{(k)}) +  
[ \nabla^{2} \calL(\btheta^{(k)}, \lambda^{(k)}) ]   \, \bd \,. \notag 
\ee 
By setting the gradient to zero, we obtain the Newton iteration update as the solution to 
\be
[ \nabla^{2} \calL(\btheta^{(k)}, \lambda^{(k)}) ]   \,  \bd   = - \nabla \calL(\btheta^{(k)}, \lambda^{(k)})  \notag 
\ee
or
\be
\left[\begin{array}{cc} 
\bH_{\lambda^{(k)}}(\btheta^{(k)})  &  \bg_c(\btheta^{(k)}) \\
 \bg_c^T(\btheta^{(k)})  & 0 \\
\end{array} \right] \, 
\left[\begin{array}{c} \bd_{\theta} \\ d_{\lambda} \end{array}\right] 
=  - \left[\begin{array}{c} 
\bg_f(\btheta^{(k)}) + \lambda^{(k)} \bg_c(\btheta^{(k)}) \\
c(\btheta^{(k)})  - \delta^2
\end{array}\right] \label{eq_SQP_up1}
\ee
where $\bg_f(\btheta)$ and $\bg_c(\btheta)$ are gradients of the objective and constraint functions with respect to $\btheta$. 
The Hessian $\bH_{\lambda^{(k)}}(\btheta^{(k)})$ is computed as  
\be
\bH_{\lambda^{(k)}}(\btheta^{(k)}) &=& \nabla^{2} f(\btheta^{(k)})+  \lambda^{(k)} \, \nabla^{2} c(\btheta^{(k)})  . \label{eq_linsys_augLagrangian}
\ee
The solution in (\ref{eq_SQP_up1}) is also minimiser to the following QP subproblem 
\be
\min \quad & \frac{1}{2}  \bd_{\theta}^T  \, \bH_{\lambda^{(k)}}(\btheta^{(k)}) \, \bd_{\theta} +    \bg_f^T(\btheta^{(k)}) \, \bd_{\theta}  \label{eq_sqp_2} \\
 \quad \text{s.t.} \quad & c(\btheta^{(k)})  +   \bg_c^T(\btheta^{(k)}) \, \bd_\theta  = \delta^2 \notag \, .
\ee
Solving either (\ref{eq_SQP_up1}) or the QP in (\ref{eq_sqp_2}) gives us the new search direction and the Lagrange multiplier, $\btheta^{(k+1)} = \btheta^{(k)} + \bdelta_{\theta}$ and $\lambda^{(k+1)} = \lambda^{(k)} + \delta_{\lambda}$.
Further details on the SQP method can be found in \cite{boggs_tolle_1995,Fletcher_book}.

We next derive the Hessian and gradient of the Lagrangian, then present a fast method to solve the linear system in (\ref{eq_SQP_up1}). A similar method was used in \cite{Phan_fLM,PetrfLMnr6}.

Following the Hessian of the objective in (\ref{eq_Hessian_obj}) 
in Appendix~\ref{sec:gradHess_ftheta} and the constraint function  in (\ref{eq_Hessian_cp}) in Appendix~\ref{sec::gradHessian_c_theta}, the Hessian $\bH_{\lambda}$ can be expressed as a rank-$R^2$  adjustment form as 
\be
\bH_{\lambda} &=& \bH_f + \lambda \, \bH_c 	\notag \\
&=& \tilde{\bG} 
+ \tilde{\bZ} \, \bPsi_{\lambda} \, \tilde{\bZ}^T
\ee
where $\bPsi_{\lambda}  = \bP_{R,R} \dvec({\bGamma_{\lambda}})$,  $\bP_{R,R}$ is a permutation matrix which maps $\vtr{\bX_{R\times R}} = \bP_{R,R} \vtr{\bX_{R\times R}^T}$, and $\tilde{\bG}$ is a block diagonal matrix of  square matrices, $\tilde{\bG}_n$, of size $I_n R \times I_n R$
\be
\tilde{\bG}_n &=& \bGamma_{\lambda}^{(-n)} \, \otimes \bI_{I_n}
- \tilde{\bZ}_{n} \, \bPsi_{\lambda} \, \tilde{\bZ}_{n}^T \,, 
\ee
$\bGamma_{\lambda}$ and $\bGamma_{\lambda}^{(-n)}$ are square matrices of size $R\times R$ respectively adjusted from the matrices $\bGamma$ and $\bGamma_{-n}$ as 
\be
\bGamma_{\lambda}^{(-n)}(r,s) &=& 
\begin{cases}
  \lambda \, \bGamma_{-n}(r,s)  \, ,	\quad r \neq s \, ,\\
  (  \lambda+1) \, \bGamma_{-n}(r,r)  \, , 	\quad r = s\,,
  \end{cases} \\
\bGamma_{\lambda}(r,s) &=& 
\begin{cases}
  \lambda \, \bGamma(r,s)   \,,	\quad r \neq s\, ,\\
  (\lambda+2) \, \bGamma(r,r)  \,, 	\quad r = s.
  \end{cases} 
\ee
Here $\dvec(\bK) = \diag(\vtr{\bK})$.
With the condition $R < \sum_{n} I_n$, the matrix $\tilde{\bZ}$ is a tall matrix of size $R(\sum_{n} I_n) \times R^2$.
$\tilde{\bG}$ is a block diagonal matrix, inverse of this matrix is efficiently computed through the inverses of its block matrices $\tilde{\bG}_n$. However, since the Hessian matrix is rank-deficient, we suggest to increase its diagonal by a sufficiently large $\mu$ to make the smallest eigenvalue positive. This is similar to adding the term $\mu \|\btheta\|_F^2$ into the objective function $f(\btheta)$
\be
\min \quad f(\btheta) + \mu \|\btheta\|_2^2 \qquad \text{s.t.} \quad c(\btheta)  \le \delta^2 \notag  
\ee
or minimising the problem (\ref{eq_cp_boundnorm2b}) with an additional constraint $\|\btheta\|_2^2 \le \alpha^2$
\be
\min \quad f(\btheta)  \qquad \text{s.t.} \quad c(\btheta)  \le \delta^2  , \quad  \|\btheta\|_2^2  \le \alpha^2 \, \notag  \,.
\ee

Since shifting eigenvalues does not change the low-rank adjustment structure of the Hessian, following \cite{Phan_fLM,PetrfLMnr6}, we can invert the damped Hessian $\bH_{{\lambda},\mu} = \bH_{\lambda} + \mu \bI$ as follows 
\be
\bH_{\lambda,\mu}^{-1} &=& (\tilde{\bG}_\mu  + \tilde{\bZ} \, \bPsi_{\lambda} \, \tilde{\bZ}^T)^{-1} \notag \\
&=& \tilde{\bG}_\mu^{-1}  - \tilde{\bG}_\mu^{-1} \tilde{\bZ} \,  (\bPsi_{\lambda}^{-1} + \tilde{\bZ}^T \tilde{\bG}_\mu^{-1}   \tilde{\bZ})^{-1} \,  \tilde{\bZ}^T\tilde{\bG}_\mu^{-1}
\label{eq_inverse_H}
\ee
where the block diagonal matrix $\tilde{\bG}_\mu = \tilde{\bG} + \mu \bI$ and 
\be
\tilde{\bG}_\mu^{-1} &=& \blkdiag(\ldots, (\tilde{\bG}_n + \mu \bI)^{-1}, \ldots)  \notag \, ,\\
(\tilde{\bG}_n + \mu \bI)^{-1} &=& ((\bGamma_{\lambda}^{(-n)} + \mu \bI) \, \otimes \bI_{I_n}
- \tilde{\bZ}_{n} \, \bPsi_{\lambda} \, \tilde{\bZ}_{n}^T )^{-1}\,.
\label{eq_inverse_Gamma_n_mu}
\ee
If $R < I_n$, the matrices $\tilde{\bZ}_{n}$ are tall and of size $RI_n \times R^2$, the inversion $(\tilde{\bG}_n + \mu \bI)^{-1}$ can  be performed even more efficiently as
\be
(\tilde{\bG}_n + \mu \bI)^{-1} &=& \bPhi_{n} \otimes \bI_{I_n}  -  (\bPhi_{n} \otimes \bI_{I_n})\, \tilde{\bZ}_{n} (\bPsi_\lambda^{-1}  +  \tilde{\bZ}_{n}^T\,  (\bPhi_{n} \otimes \bI_{I_n})\ \,  \tilde{\bZ}_{n})^{-1}\, \tilde{\bZ}_{n}^T\, (\bPhi_{n} \otimes \bI_{I_n}) \notag \\
&=& \bPhi_{n} \otimes \bI_{I_n}  -  (\bPhi_{n} \otimes \bU^{(n)}) \dvec(\1\oslash {\bGamma_n})   \notag \\ && \quad (\bPsi_\lambda^{-1}  +  \dvec(\1\oslash {\bGamma_n})  (\bPhi_{n} \otimes \bGamma_n)\, \dvec(\1\oslash {\bGamma_n}))^{-1}\,   \notag \\ && \quad\dvec(\1\oslash {\bGamma_n})  \, (\bPhi_{n} \otimes \bU^{(n) T})  \notag \\
&=& \bPhi_{n} \otimes \bI_{I_n}  -  (\bPhi_{n} \otimes \bU^{(n)})     ( \bP_{R,R} \dvec({\bGamma_{n}^2\oslash\bGamma_{\lambda}}) +   \bPhi_{n} \otimes \bGamma_n)^{-1}  \, (\bPhi_{n} \otimes \bU^{(n) T})  \label{eq_inverse_Gamma_n_mu_lowcost}
\ee
where $\bPhi_{n}  = (\bGamma_{\lambda}^{(-n)} + \mu \bI)^{-1}$, and $\oslash$ represents the Hadamard element-wise division.
The last expression is obtained by using the following identity  whose proof is provided in Appendix~\ref{sec_Prr_identity}
\be
\bP_{R,R} \, \dvec(\bGamma_n) =  \dvec(\bGamma_n) \, \bP_{R,R} \, . \label{eq_Prr_identity}
\ee
%

Now, by exploiting the rank-1 expansion, and replacing $\bH_{\lambda}$ by the damped $\bH_{\lambda,\mu}$, 
inverse of the Hessian $\nabla^{2} \calL$ in (\ref{eq_SQP_up1}) can be expressed as 
\be
\left[\begin{array}{cc} 
\bH_{\lambda,\mu}  &  \bg_c \\
 \bg_c  & 0 \\
\end{array} \right]^{-1} = \left[\begin{array}{cc} 
\bH_{\lambda,\mu}^{-1}  \\
  & 0 \\
\end{array} \right]  - 
\frac{1}{\bg_c^{T}  \bH_{\lambda,\mu}^{-1} \bg_c}
\left[\begin{array}{cc} 
\bH_{\lambda,\mu}^{-1} \bg_c \\
-1
\end{array} \right]  \, \left[\begin{array}{cc} 
\bg_c^T \bH_{\lambda,\mu}^{-1}  & 
-1
\end{array} \right] \label{eq_invH_Lagrangian}\,
\ee  
then we apply the inversion in (\ref{eq_inverse_H}) with (\ref{eq_inverse_Gamma_n_mu}) or with (\ref{eq_inverse_Gamma_n_mu_lowcost}) to compute the inverse of the Hessian. 

%
From (\ref{eq_SQP_up1}), new estimates of the Lagrange multiplier $\lambda^{(k+1)}$ and search direction $\bd_{\theta}$ are given by
\be
	 \lambda^{(k+1)} &=& \frac{c(\btheta^{(k)}) - \delta^2 - \bg_c^T(\btheta^{(k)})  \, \bH_{\lambda,\mu}^{-1}(\btheta^{(k)})   \bg_f(\btheta^{(k)})  }{\bg_c^T(\btheta^{(k)})  \, \bH_{\lambda,\mu}^{-1}(\btheta^{(k)}) \, \bg_c(\btheta^{(k)}) }     \,  ,\\
	\bd_{\theta} &=& -\bH_{\lambda,\mu}^{-1}(\btheta^{(k)})  (\bg_f(\btheta^{(k)})  +  \lambda^{(k+1)} \, \bg_c(\btheta^{(k)}) )    \, ,
\ee
and we obtain a new estimate of parameters $\btheta^{(k+1)} = \btheta^{(k)} + \bd_{\theta}$. In addition, the loading components of rank-1 tensors are then normalised to have equivalent norm after each iteration, i.e.,
\be
	\bu_r^{(n)} \leftarrow  \bu_r^{(n)} \frac{\eta_r^{1/N}}{\|\bu_r^{(n)}\|}
\ee
where $\eta_r = \prod_n   \|\bu_r^{(n)}\|$.

%
%

\subsection{Implementation}\label{sec:implement_cp_bound}

A requirement for the EPC methods of rank-1 tensors is that the initial point is feasible, i.e., obeys the constraint $\| \tY - \hat{\tY}\|_F  \le \delta$. 
In practice, we apply the correction method after fitting the tensor following the ordinary CP model. By this way, the current estimated tensor is a feasible point, where $\delta = \| \tY - \hat{\tY}\|_F$ is the current approximation error. Since the ACEP algorithm solves sub-problems in closed-form, the new update points are always in a feasible region, i.e., $c(\theta) \le \delta^2$, while the objective function decreases sequentially or at least is kept to not increase. However, unlike the ACEP, the SQP algorithm solves the sub-problems which approximate the main problem, even when provided an initial feasible point, this algorithm still needs to seek a feasible region in some first iterations. Hence, newly updated points may not remain in the feasible region. Our experience is that we execute the ACEP in some small number of iterations, then switch to the SQP or Interior Point (ITP) algorithm. 

In practice, one can gradually increase the bound, $\delta$, e.g., by a factor of 1.1, until the norm of rank-1 tensors attains the desired value. The CPD with EPC can be  implemented as in 
Algorithm~\ref{alg_CP_with_normcorrection}.

 \setlength{\algomargin}{1em}
\begin{algorithm}[t!]
\SetFillComment
\SetSideCommentRight
\CommentSty{\footnotesize}
\caption{{\tt{CPD with EPC}}\label{alg_CP_with_normcorrection}}
\DontPrintSemicolon \SetFillComment \SetSideCommentRight
\KwIn{Data tensor $\tY$:  $(I_1 \times I_2 \times \cdots \times I_N)$,  and a rank $R$ and error bound $\delta$} 
\KwOut{$\tX = \llbracket \boldeta; \bU^{(1)}, \bU^{(2)}, \ldots, \bU^{(N)}\rrbracket $ of rank $R$}
\SetKwFunction{cpd}{CPD} 
\SetKwFunction{bals}{bals} 
\Begin{
\nl Initialize $\tX_0 = \llbracket \boldeta_0; \bU_0^{(1)}, \bU_0^{(2)}, \ldots, \bU_0^{(N)}\rrbracket$ \;
\Repeat{a stopping criterion is met}{
\nl \If {$\|\boldeta_{k-1}\|_2^2 \ge \delta^2$ } {
	\nl Solve (\ref{eq_cp_boundnorm}) to find a tensor $\tX_{k}$ such that $\min \quad \|\boldeta_{k}\|_2^2$  \quad s.t. \,\,$\|\tY - \tX_{k}\|_F \le \|\tY - \tX_{k-1}\|_F$\;
	}
\nl \Else
{\nl	 Seek $\tX_k$ such that  $\min \|\tY - \tX_k\|_F^2$ with the initial $\tX_{k-1}$
}
}
}
\end{algorithm}

%

\section{Canonical Polyadic Tensor Decomposition with Bound on Norm of Rank-1 Tensors}\label{sec:cp_bounded_rank1norm} 

In contrast to the CP decomposition which seeks a tensor approximation having a minimal norm of rank-1 tensors, in this section we consider a constrained tensor approximation, in which the norm of rank-1 tensors is bounded
\be
\min \quad &\| \tY - \hat{\tY} \|_F^2 \quad 
\text{s.t.} \quad & \|\boldeta \|_2^2 \le \epsilon^2  \, .\label{eq_cp_bound_lda}
\ee

\subsection{Alternating update algorithm}\label{r[fv}

Similar to the previous section, we can absorb $\boldeta$ into a factor matrix $\bU^{(n)}$ and rewrite the above optimization problem as  
\be
\min \quad &\| \tY - \hat{\tY} \|_F^2 \quad 
\text{s.t.} \quad & \|\bU_\eta^{(n)}\|_F^2 \le \epsilon^2  \, \label{eq_cp_bound_Ulda}\,,
\ee
or as a Quadratic programming with a bounded norm for $\bU_{\eta}^{(n)}$
\be
\min \quad & \tr(\bU_{\eta}^{(n)}  \bGamma_{-n} \bU_{\eta}^{(n) T}) - 2 \tr(\bG_n \, \bU_{\eta}^{(n) T})     \, \label{eq_cp_bound_Ulda_2}\\
\text{s.t.} \quad & \|\bU_\eta^{(n)}\|_F^2 \le \epsilon^2  \, \notag .
\ee
The above equation is followed the expansion of the Frobenius norm in (\ref{eq_expansion_of_error}). Next we convert the above matrix-variate QP to the one for a vector of length $R \times 1$. 

Let $\bGamma_{-n} = \bV_n \bSigma \bV_n^T$ be the eigenvalue decomposition of $\bGamma_{-n}$. Then denote $\bF_n = \bG_n \bV_n$, and $\bZ_n = \bU_{\eta}^{(n)} \bV_n $. The optimization in (\ref{eq_cp_bound_Ulda_2}) is transformed into 
\be
\min \quad & \tr(\bZ_n \bSigma_n \bZ_n^T) - 2 \tr(\bF_n \, \bZ_n^T)     \, \label{eq_cp_bound_Z}\\
\text{s.t.} \quad & \|\bZ_n\|_F^2 \le \epsilon^2  \, \notag .
\ee
Similar to (\ref{eq_Z_SCQP}) and according to Lemma~\ref{lem_sqp_simplify} in Appendix~\ref{sec:scqp_identical_eig}, the minimiser $\bZ_n^{\star}$ to the matrix-variate QP in (\ref{eq_cp_bound_Z}) can be derived from the minimiser $\bz^{\star}$ to the following constrained QP
\be
\min \quad & \bz^T \bSigma \bz - 2 \bc^T  \, \bz  \, \quad  
\text{s.t.} \quad \bz^T\bz  \le \epsilon^2  \, \label{eq_cp_bound_Ulda_b}
\ee
where the vector $\bc = [\ldots, \|\vf_r^{(n)}\|_2, \ldots]^T$ comprises the norm of columns of $\bF_n$. We note that the above  QP problem with an inequality constraint can be solved in closed-form.

If $\sum_{r = 1}^{R} \frac{c_r^2}{\sigma_r^2} \le \epsilon^2$, the minimiser to (\ref{eq_cp_bound_Ulda_b}) is simply the point  
\be
\bz^{\star} = \bc \oslash \bsigma\,. \label{eq_bals_z_case1}
\ee
This case often happens when the current parameter point is in a feasible set, and the bound is set to a relatively high value. 

Otherwise, $\bz^{\star}$ is a minimiser to the QP over a sphere which again can also be solved in closed-form\cite{Phan_QPS}
\be
\min \quad & \bz^T \bSigma \bz - 2 \bc^T  \, \bz  \, \quad  
\text{s.t.} \quad \bz^T\bz  = \epsilon^2  \, \label{eq_cp_bound_Ulda_sphere} \,.
\ee

The proposed algorithm works in the same manner as the ordinary ALS algorithm. We call this the BALS algorithm.

\subsection{Relation between BALS and ALS with smoothness constraint}

We consider the case when the last column of $\bF_n$ is non zero, i.e., $c_R \neq 0$. Assuming that the eigenvalues of $\bGamma_{-n}$ are ordered in the descending order, i.e., $\sigma_1 \ge \sigma_2 \ge \cdots \ge \sigma_R>0$, the SCQP in (\ref{eq_cp_bound_Ulda_sphere}) has a minimiser given in form of 
\be
\bz^{\star} &=&  [\ldots, \frac{c_r}{ \sigma_r  - \widetilde{\lambda} } , \ldots ]  
\ee
where $  \widetilde{\lambda} $ is a unique solution in $[\sigma_R - \|\bc\|,  \sigma_R - \|\bc\| (1 - 1/\epsilon)]$ of the secular equation which can be solved in closed-form\cite{GANDER1989815,Phan_QPS}
\be
{\bz^{\star}}^T\, \bz^{\star}  = \sum_r  \frac{c_r^2}{(\sigma_r - \widetilde{\lambda})^2}= \epsilon^2 \, .\notag
\ee
The minimiser in (\ref{eq_bals_z_case1}) is a particular case of the above when $\widetilde{\lambda}  = 0$.
Hence, from the conversion of QP for matrix-variate in Appendix~\ref{sec:sqp_matrixvariate} and Lemma~\ref{lem_sqp_simplify}, 
we can write $\bZ_n^{\star}$ as 
\be
\bZ_n^{\star} =  {\bF}_n\,  \diag( \ldots, (\sigma_r - \widetilde{\lambda})^{-1}, \ldots )\, .
\ee
Replacing this into $\bU_{\eta}^{(n)}$, we obtain an update rule   
\be
\bU_{\eta}^{(n)}&=& \bG_n \,\bV_n  \diag( \ldots, (\sigma_r - \widetilde{\lambda})^{-1}, \ldots ) \bV_n^T \,
\notag \\ 
&=& \bG_n (\bGamma_{-n} - \widetilde{\lambda} \, \bI_R )^{-1}  \label{eq_bals_nonzero_cR} \,.
\ee
The above update rule (\ref{eq_bals_nonzero_cR}) is indeed similar to the ALS update rule derived for the objective function in (\ref{eq_cpd_smooth}) with $\mu = - \widetilde{\lambda}$. Here,  we show a relation between the regularization parameter $\mu$ and the bound $\epsilon$. In the decomposition in (\ref{eq_cpd_smooth}), the regularisation or damping parameter $\mu$ is often fixed or adjusted to keep the cost function non-increasing. In our algorithm, the parameter $\widetilde{\lambda}$ is a root of a secular equation, and is updated in each iteration. 

\subsection{Sequential quadratic programming method}

Similar to the SQP algorithm for the optimization problem in (\ref{eq_cp_boundnorm2b}), we relax the unit-length constraints of the loading components and develop an SQP algorithm for the CPD with bounded rank-1 tensor norm 
\be
\min \quad c(\btheta)  \quad \text{s.t.} \quad f(\btheta) \le \epsilon^2 \, \label{eq_cp_bound_rank1norm}
\ee
where the functions $f(\btheta)$ and $c(\btheta)$ exchange their roles in the optimisation (\ref{eq_cp_boundnorm2b}).
The Lagrangian to the above constrained optimisation is given by 
\be
\calL(\btheta, \lambda) =  c(\btheta) + \lambda \, (f(\btheta)  - \epsilon^2) \,.
\ee
Similar to the Lagrangian in (\ref{eq_augLagragian_boundnorm1}), 
the new search direction is minimiser to the QP subproblem
\be
\min \quad & \frac{1}{2}  \bd_{\theta}^T  \, \widetilde{\bH}_{\lambda^{(k)}}(\btheta^{(k)}) \, \bd_{\theta} +   \bg_{\lambda^{(k)}}^T(\btheta^{(k)}) \, \bd_{\theta}  \label{eq_sqp_bcpd} \\
 \quad \text{s.t.} \quad & f(\btheta^{(k)})  +   \bg_f^T(\btheta^{(k)})\,\bd_\theta  \le \epsilon^2 \notag \, , 
\ee
where 
\be
\widetilde{\bH}_{\lambda}(\btheta)  &=&  \nabla^2 c(\btheta)  +  \lambda \, \nabla^2 f(\btheta)  \notag \, ,\\
\bg_{\lambda}(\btheta) &=& \bg_c(\btheta)  + \lambda  \bg_f(\btheta)  \notag  \,.
\ee
If the non-constrained solution, $-\widetilde{\bH}_{\lambda}^{-1} \,  \bg_{\lambda}$, is in the feasible set, it is the minimiser and $d_{\lambda} = 0$. Otherwise, we solve the QP with an equality constraint. Similar to (\ref{eq_sqp_2}) it leads to find the solution to a system 
\be
\left[\begin{array}{cc} 
\widetilde{\bH}_{\lambda}  &  \bg_f \\
 \bg_f^T & 0 \\
\end{array} \right] \, 
\left[\begin{array}{c} \bd_{\theta} \\ d_{\lambda} \end{array}\right] 
=  - \left[\begin{array}{c} 
\bg_c + \lambda  \bg_f  \\
f - \epsilon^2
\end{array}\right] \label{eq_SQP_bcdp_up2}
\ee
where the Hessian $\widetilde{\bH}_{\lambda} = \lambda \, \bH_{1/{\lambda}} $ shares the low-rank adjustment structure of $\bH_{\lambda}$.
This finally leads to the compact update rules for the Lagrange multiplier and the search direction 
\be
	 \lambda^{(k+1)} &=& \frac{f(\btheta^{(k)}) - \epsilon^2- \bg_f^T(\btheta^{(k)})  \, \widetilde{\bH}_{\lambda,\mu}^{-1}(\btheta^{(k)})   \bg_c(\btheta^{(k)})  }{\bg_f^T(\btheta^{(k)})  \, \widetilde{\bH}_{\lambda,\mu}^{-1}(\btheta^{(k)}) \, \bg_f(\btheta^{(k)}) }     \,  ,\\
	\bd_{\theta} &=& -\widetilde{\bH}_{\lambda,\mu}^{-1}(\btheta^{(k)})  (\bg_c(\btheta^{(k)})  +  \lambda^{(k+1)} \, \bg_f(\btheta^{(k)}) )    \, .
\ee


\section{Numerical Results}\label{sec::simulation}

\begin{example}[CP with and without EPC of rank-1 tensors]\label{ex1}

We considered tensors of size $I \times I \times I$ and rank-$R = 5$, where $I = R-1$. The first four loading components were highly collinear, with $\bu^{(n) T}_r \bu^{(n)}_s  = 0.99$ for all $n$ and $1\le r \neq s \le R-1$. All components are unit-length vectors.
The factor matrices with specific correlation coefficients were generated using the subroutine $``\tt{gen\_matrix}''$  in the TENSORBOX\cite{Phan_Tensorbox}.

We decomposed the tensor using the fast Levenberg-Marquard (fLM) algorithm\cite{Phan_fLM}. The factor matrices were initialized by a matrix $[\bI_4, \1_4]$.
The algorithm did not converge even after 3000 iterations, while the norm of rank-1 tensors increased dramatically as shown in Fig.~\ref{fig_bals1}.
For this case of the decomposition, we applied the EPC after 10, 20, 50, 100 iterations of the algorithm and when the estimation stopped due to degeneracy, i.e., we sought a tensor which had a minimal norm of rank-1 tensors while still preserving $\|\tY - \hat{\tY}\|_F \le \delta_k$, where $\delta_k$ was the approximation error at the iteration $k$-th.
The results illustrated in Fig.~\ref{fig_bals1} shows that the correction lowered the norm of rank-1 tensors to 5, and helped the fLM to fully explain the tensor with the relative errors $\frac{\|\tY  - \hat{\tY}\|_F}{\|\tY\|_F} \le 10^{-7}$.



\begin{figure}[t]
\centering
\subfigure[]{\includegraphics[width=.47\linewidth, trim = 0.0cm 0cm 0cm 0cm,clip=true]{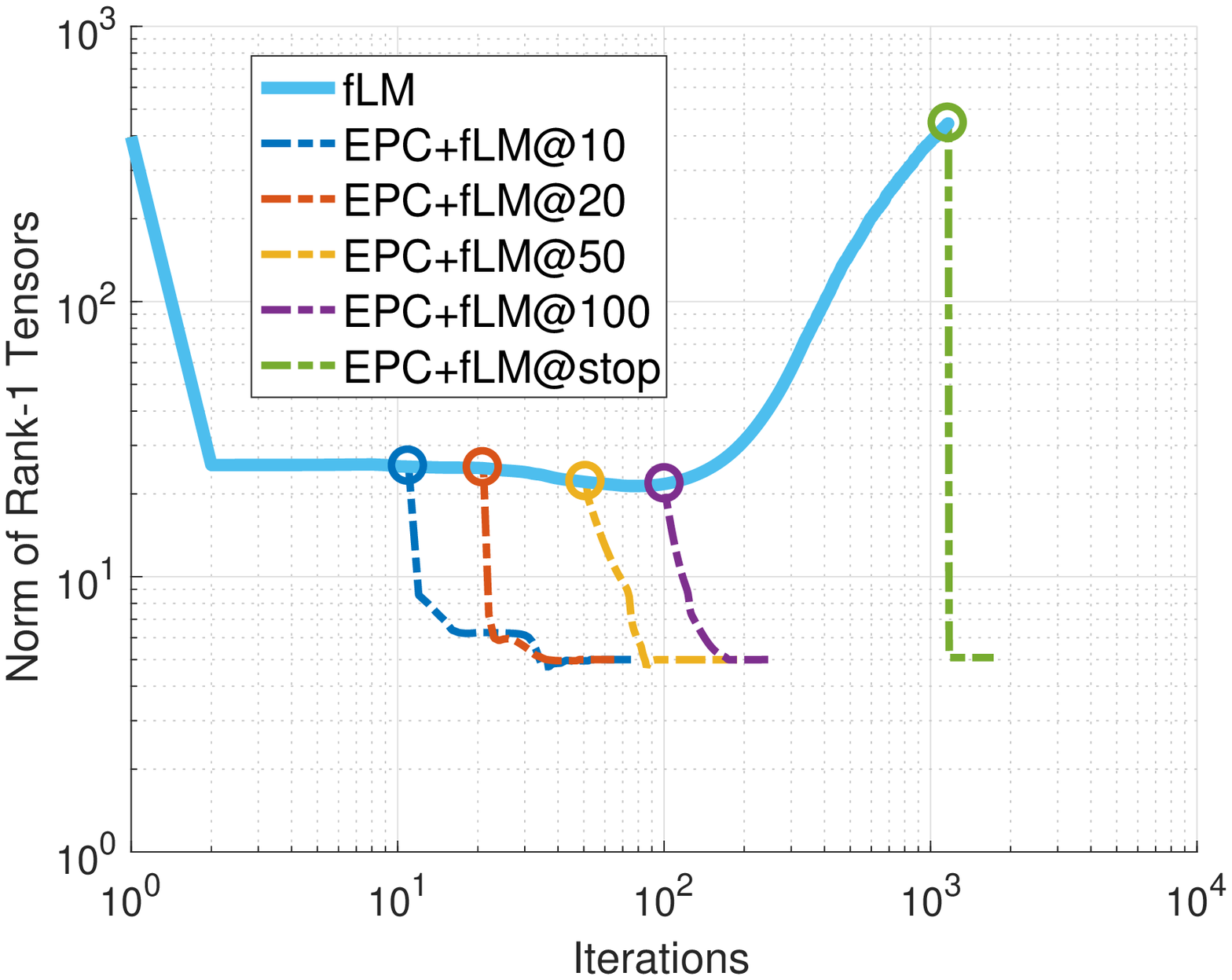}\label{fig_ex_a1}}
\subfigure[]{\includegraphics[width=.48\linewidth, trim = 0.0cm 0cm 1.47cm 0cm,clip=true]{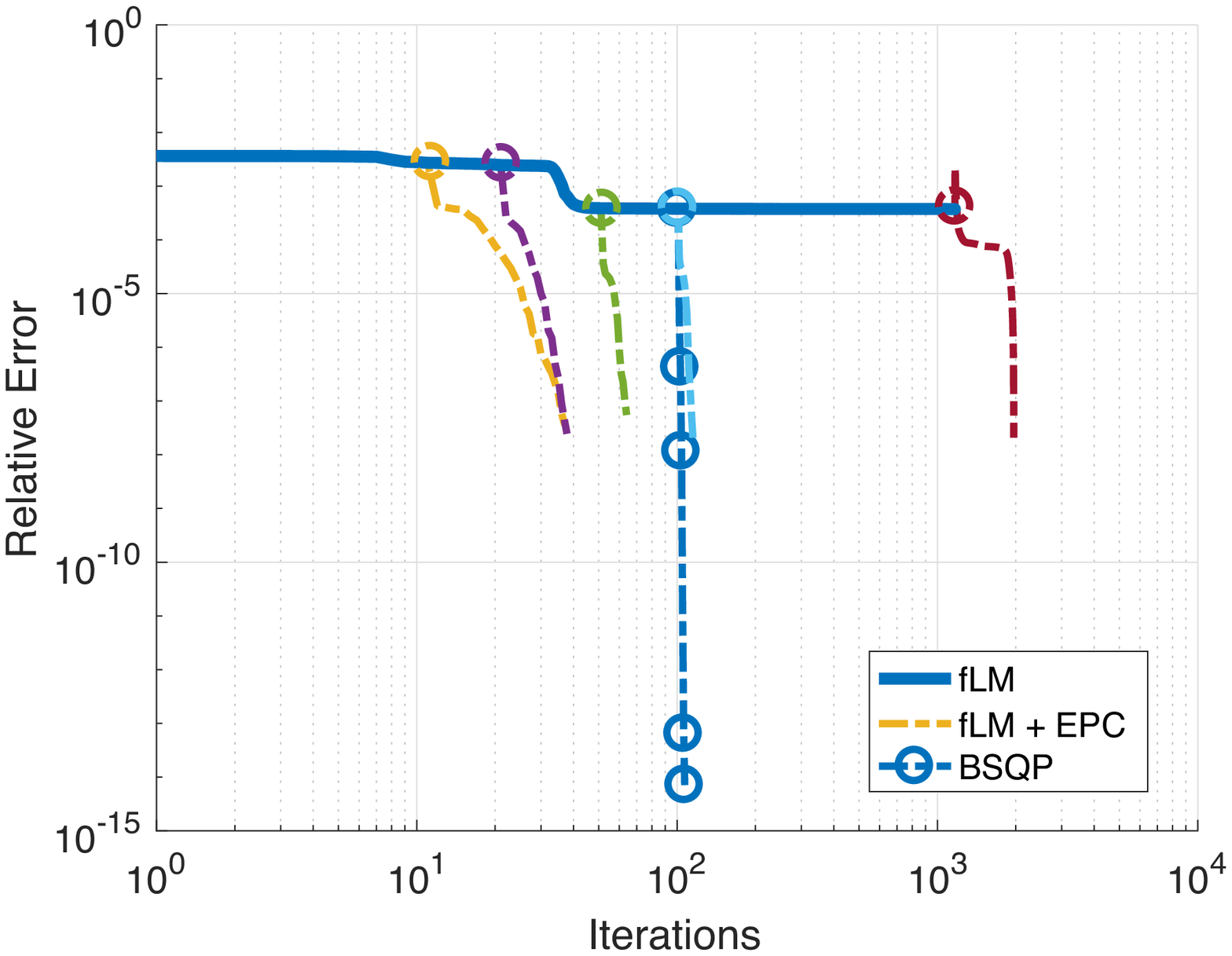}\label{fig_ex_b1}
}
\caption{Comparison of performances of the fLM with and without EPC in Example~\ref{ex1}. 
The algorithm gets stuck in false local minima, while the norm of rank-1 tensors, $\displaystyle \sum_{r=1}^{R} \eta_r^2$, increases dramatically with the number of iterations in \subref{fig_ex_a1}. 
However, the fLM converges quickly when applying EPC after 10, 20 , 50, 100 and when the fLM stops the estimation.
}
\label{fig_bals1}
\end{figure}
 
 \end{example}

\begin{example}[The case when rank-1 tensors have different weights]\label{ex1b}

We  decomposed a similar tensor as in Example~\ref{ex1}, but intensities of the rank-1 tensors were in different scales, $\eta_r = 10 r$, for $r = 1, \ldots, R$. With the same initial values, i.e., $[\bI_4, \1_4]$, but without the correction of rank-1 tensors, the fLM did not converge even when the relative error approached $10^{-4}$.
The algorithm stopped when its damping parameter increased to an extremely large value. However, when applying the rank-1 correction, e.g., at 10, 20, 50, 100 and 2000 iterations, the fLM converged in less than 1000 iterations as seen in Fig.~\ref{fig_bals1b}. 

For this tensor, we applied the BSQP algorithm to the estimated tensor using fLM after 100 iterations. 
The bound in BSQP, $\sum_{r} \|\eta_r\|_2^2 \le \epsilon^2$, was set to the norm of initial rank-1 tensors. During the estimation process, the bound, $\epsilon$, was increased by a factor of 2 if there was not a significant change in the relative error, i.e., the algorithm was getting stuck into local minima, or the relative error did not decrease.
If the relative error tends to converge, we decrease the bound, e.g., by a factor of 1.5.
 Performance of the BSQP algorithm is plotted in Fig.~\ref{fig_bals1b}. The algorithm converged and succeeded in factorizing the tensor. 
The performance of the algorithm is also confirmed in the decomposition of the tensor in Example~\ref{ex1} as illustrated in Fig.~\ref{fig_ex_b1}.

 \begin{figure}[t]
\centering
{\includegraphics[width=.48\linewidth, trim = 0.0cm 0cm 0cm 0cm,clip=true]{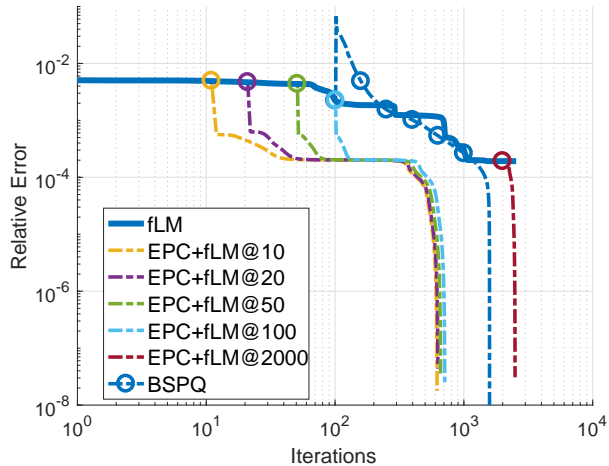}\label{fig_ex_a3}}
\caption{The fLM algorithm can decompose the tensor in Example~\ref{ex1b} with an exact fit when using the correction method of rank-1 tensors.
}
\label{fig_bals1b}
\end{figure}

\end{example}

\begin{example}[CP with a bound constraint on rank-1 tensors]\label{ex2}

In this example, we compare performances of the fastALS and fLM  algorithms for the ordinary CP,
 and the BALS and BSQP algorithms for CP with a bound constraint on rank-1 tensors. 
We decomposed the tensor in Example~\ref{ex1}, and ran the fastALS in 10 iterations to generate initial values for the four considered algorithms.
Fig.~\ref{fig_ex_a1} plots the relative errors of algorithms. 
The fLM algorithm achieved a lower relative error than ALS, but as in the previous examples, none of them could achieve an exact fit (zero approximation error).

The BALS and BSQP decomposed the tensor with an upper bound $ \sum_{r} \eta_n^2 \le 5.05$.
BALS achieved a much lower approximation error than ALS and fLM, and tent to converge with much higher number of iterations. 
The BSQP converged after around 90 iterations. 

When using with the EPC method, the two algorithms fLM and BSQP quickly converged after 10 to 20 iterations. The results also indicate that ALS combined with EPC had a similar performance to that of BALS.

\begin{figure}[t]
\centering
\subfigure[Without EPC]{\includegraphics[width=.48\linewidth, trim = 0.0cm 0cm 0cm 0cm,clip=true]{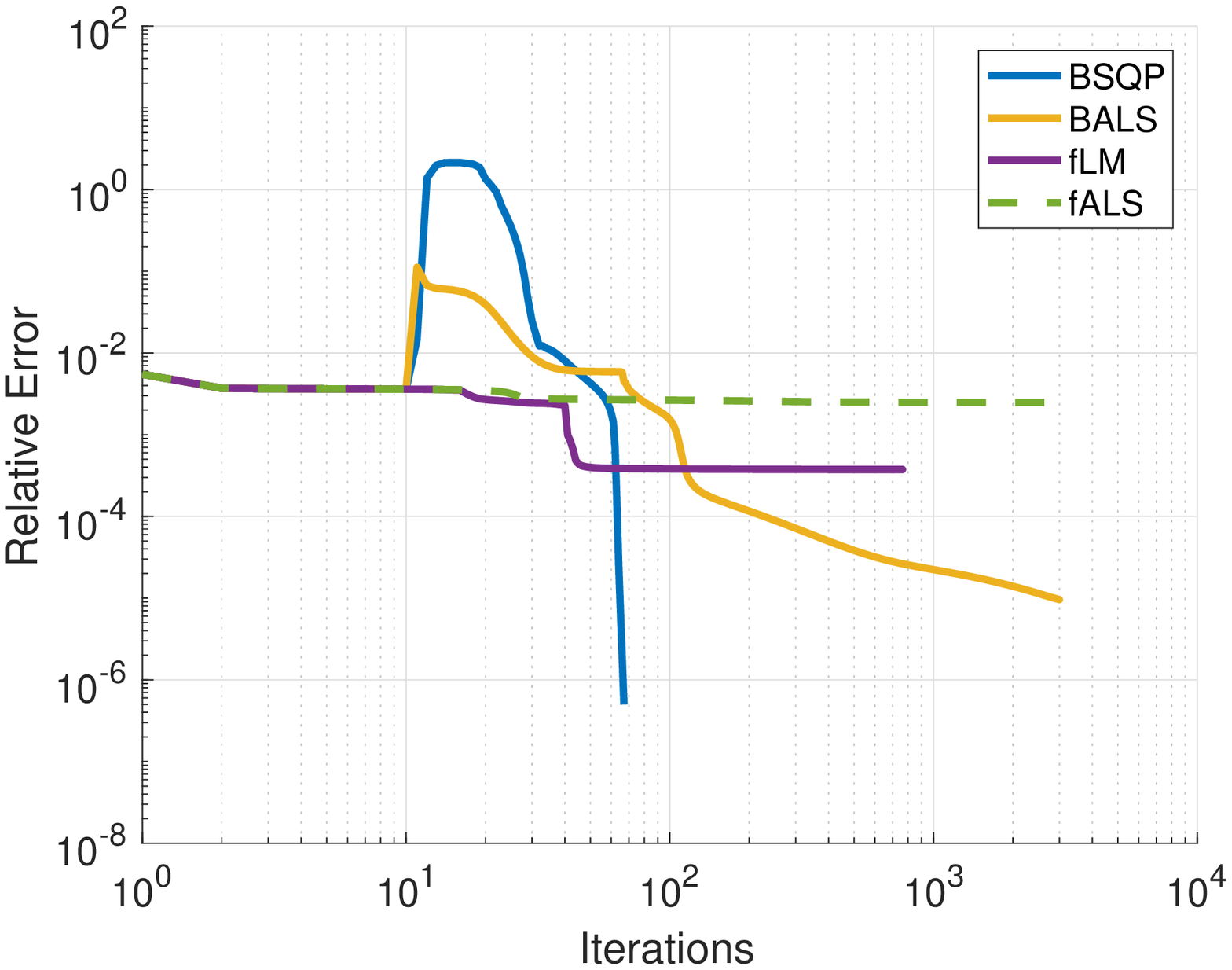}\label{fig_ex_a2}}
\subfigure[With EPC]{\includegraphics[width=.47\linewidth, trim = 0.0cm 0cm 0cm 0cm,clip=true]{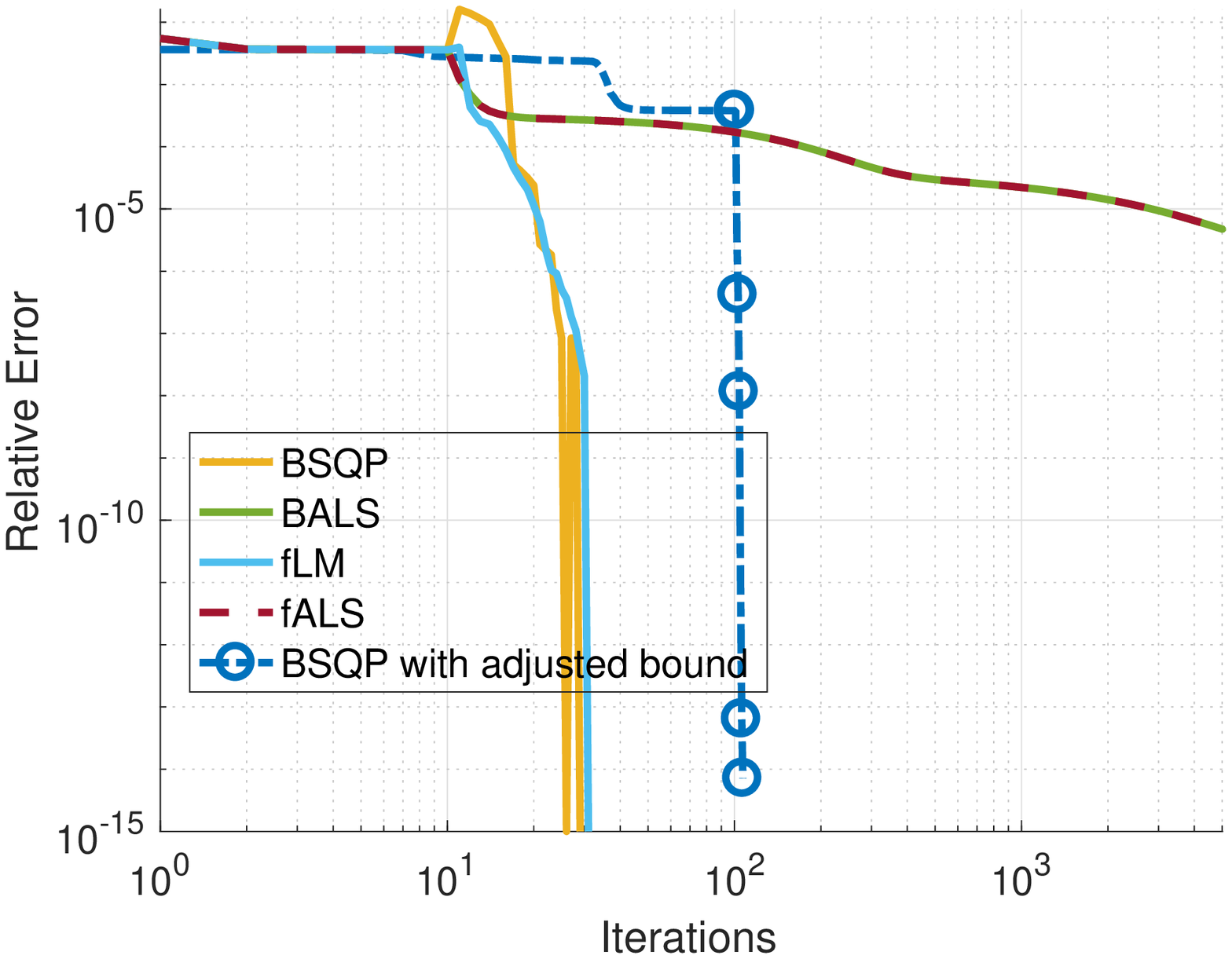}\label{fig_ex_b2}}
\caption{Decomposition of the tensor in Example~\ref{ex1} using the algorithm for CPD with bounded norm constraints.}
\label{fig_bals2}
\end{figure}

\end{example}

\begin{example}\label{ex4}

In this example, we decomposed cubic tensors with size and rank respectively given by $I_n = 4$ and $R = 5$, $I_n = 7$ and $R = 10$, $I_n = 12$ and $R = 15$. As in Example~\ref{ex1}, the first $I_n$ loading components of each factor matrices are highly correlated, while the rest $(R-I_n)$ loading components were randomly generated. A small Gaussian noise was added into the tensors. The results were reported for 150 independent runs. The factor matrices were generated as i.i.d. Gaussian with zero mean and unit variance.

Results for the noise-free cases are compared in Fig.~\ref{fig_bals3mc_2}. Success ratio at a specific error is assessed as the percentage of independent runs that an algorithm attained this error. 
For these hard decomposition scenarios, the fLM algorithm could explain the tensors with a relative error of $10^{-6}$ in about 57\% of the runs for the tensors of size $5 \times 5 \times 5$, but in less than 30\% of runs for the tensors of bigger sizes $7 \times 7 \times 7$ and $12 \times 12 \times 12$.
In most of the tests, the fLM got stuck in local minima with the relative error around $10^{-3}$. However, when using the EPC, either with ACEP or with the SQP method for ECP (SCEP), after executing 10 iterations, the success ratios of the fLM were insignificantly improved and exceeded 96\% for the relative error of $10^{-6}$.

\begin{figure}[t]
\centering
\subfigure[$I = 4$, $R = 5$]{
\begin{minipage}{.31\linewidth}
\includegraphics[width=1\linewidth, trim = 0.0cm -.4cm 0cm 0cm,clip=true]{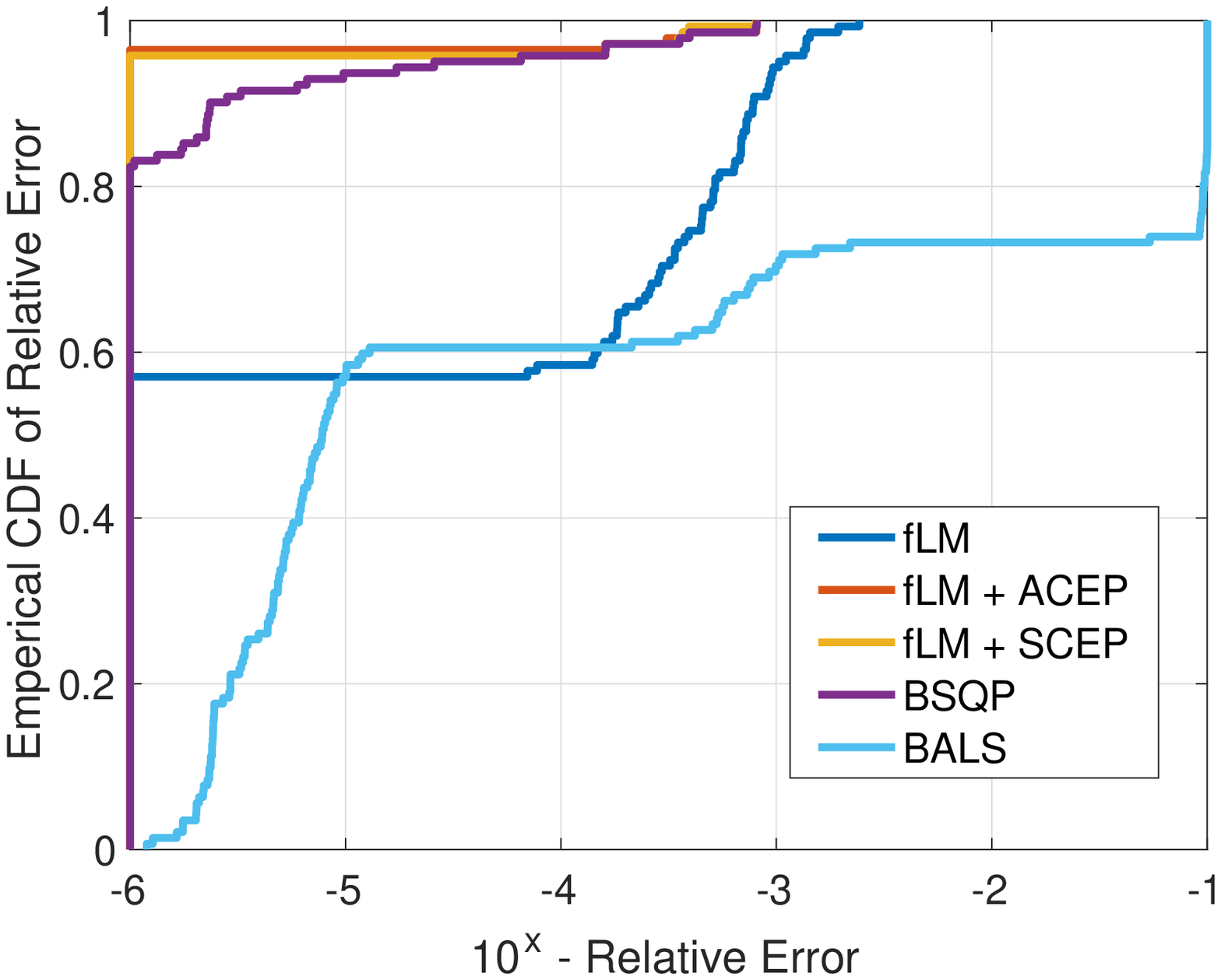}\\\includegraphics[width=1\linewidth, trim = 0.0cm -.4cm 0cm 0cm,clip=true]{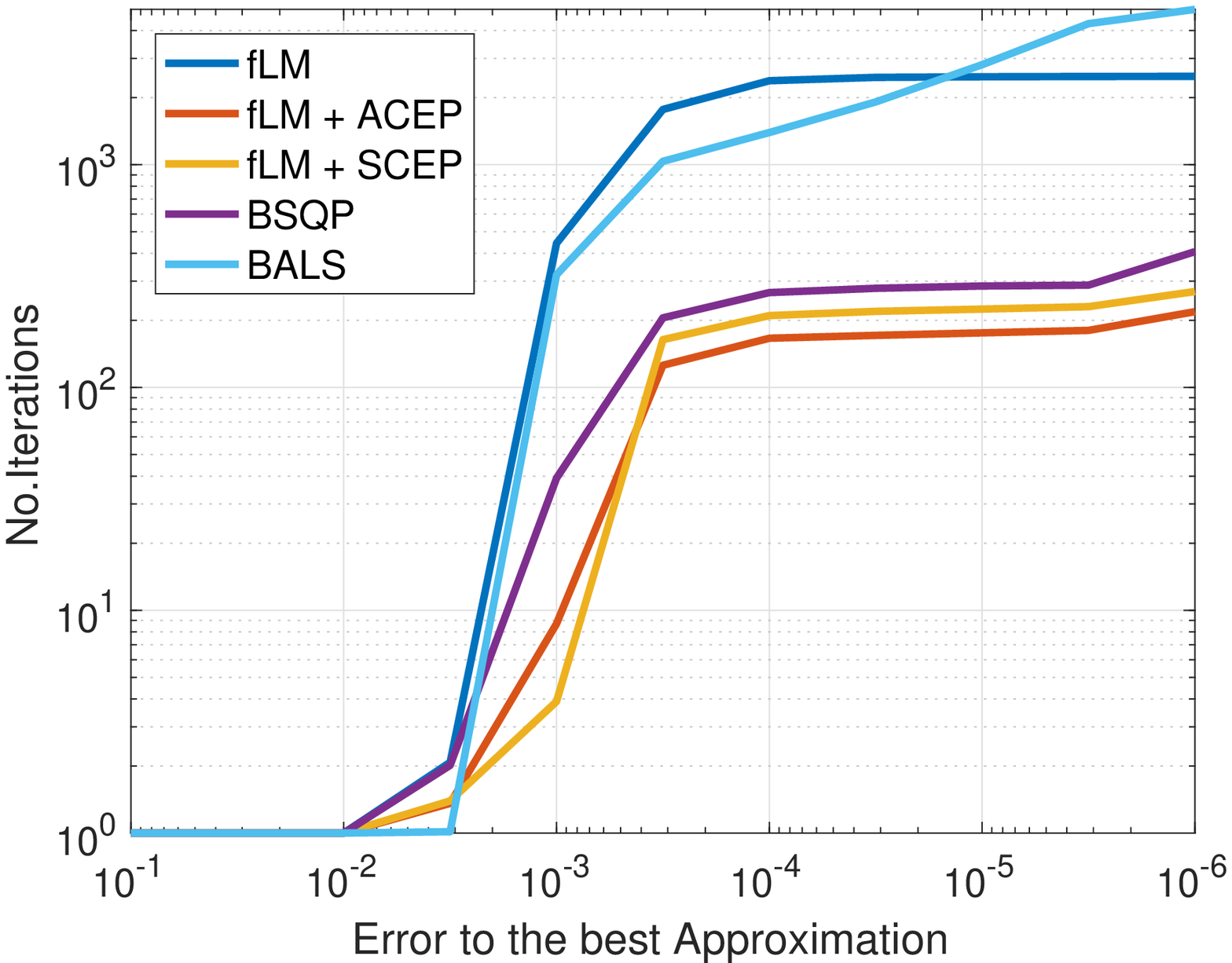}
\end{minipage}
\label{fig_ex_a4}
}
\subfigure[$I = 7$, $R = 10$]{
\begin{minipage}{.31\linewidth}
\includegraphics[width=1\linewidth, trim = 0.0cm -.4cm 0cm 0cm,clip=true]{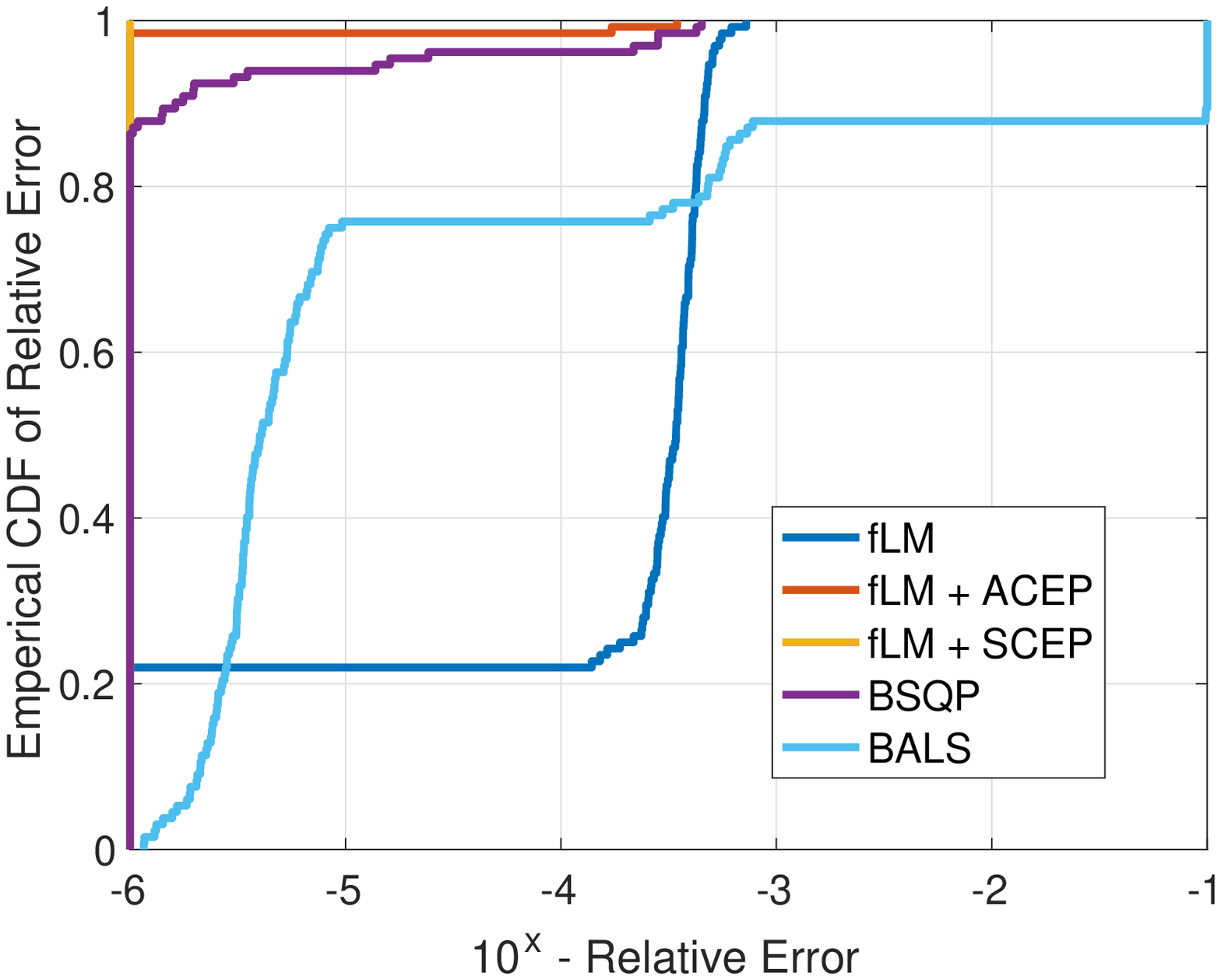}\\
\includegraphics[width=1\linewidth, trim = 0.0cm -.4cm 0cm 0cm,clip=true]{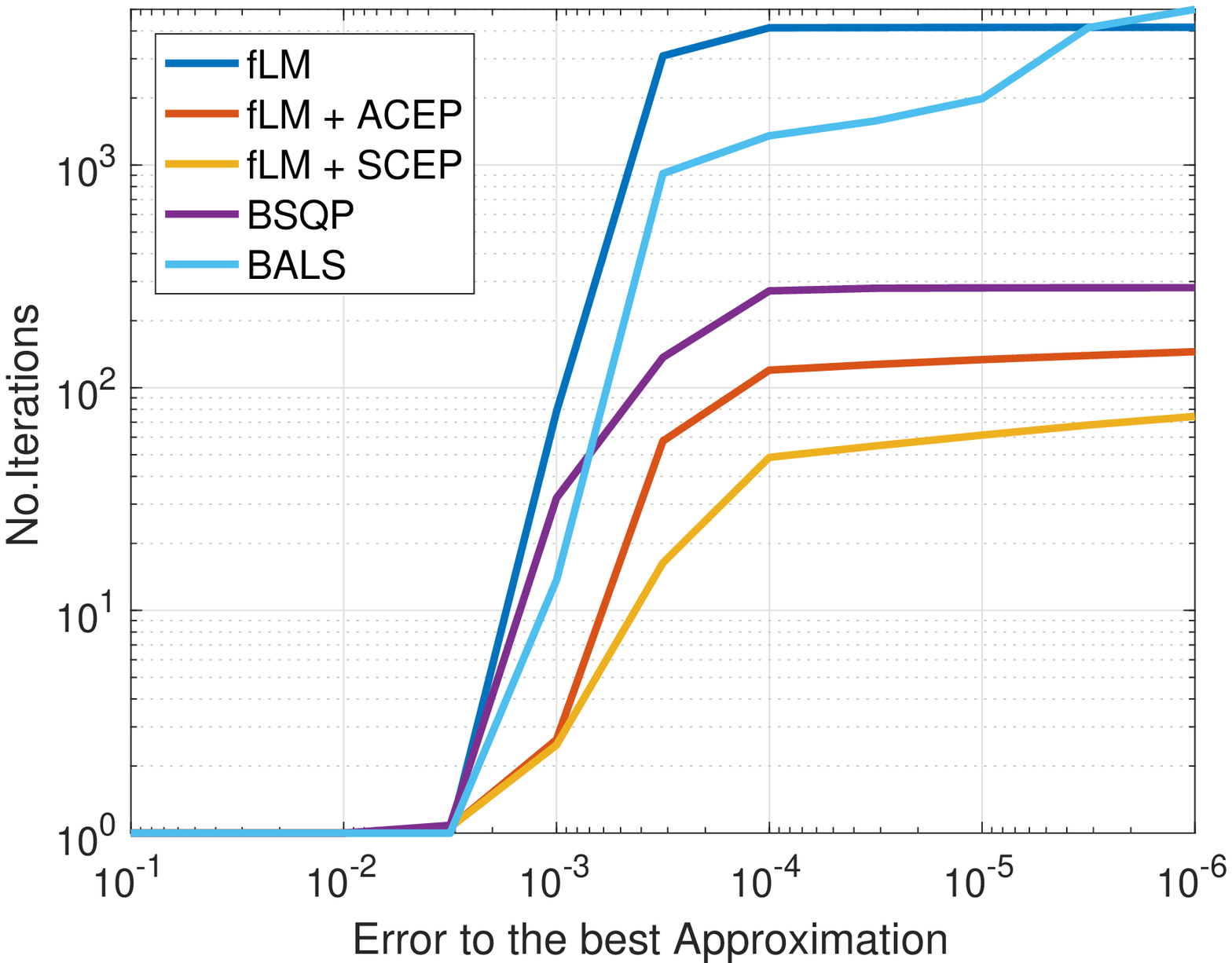}
\end{minipage}
\label{fig_ex_b4}
}
\subfigure[$I = 12$, $R = 15$]{
\begin{minipage}{.31\linewidth}
\includegraphics[width=1\linewidth, trim = 0.0cm -.4cm 0cm 0cm,clip=true]{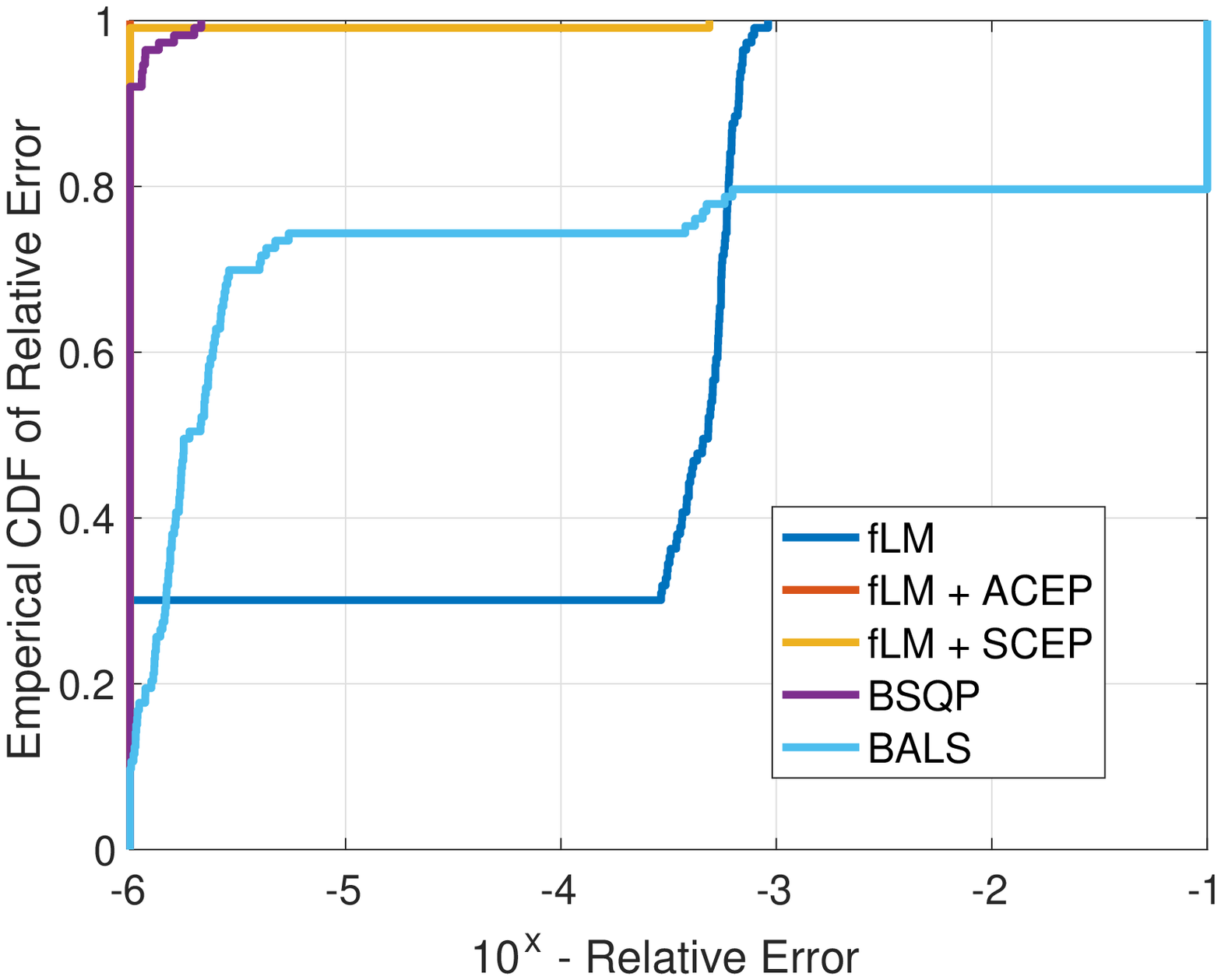}\\
\includegraphics[width=1\linewidth, trim = 0.0cm -.4cm 0cm 0cm,clip=true]{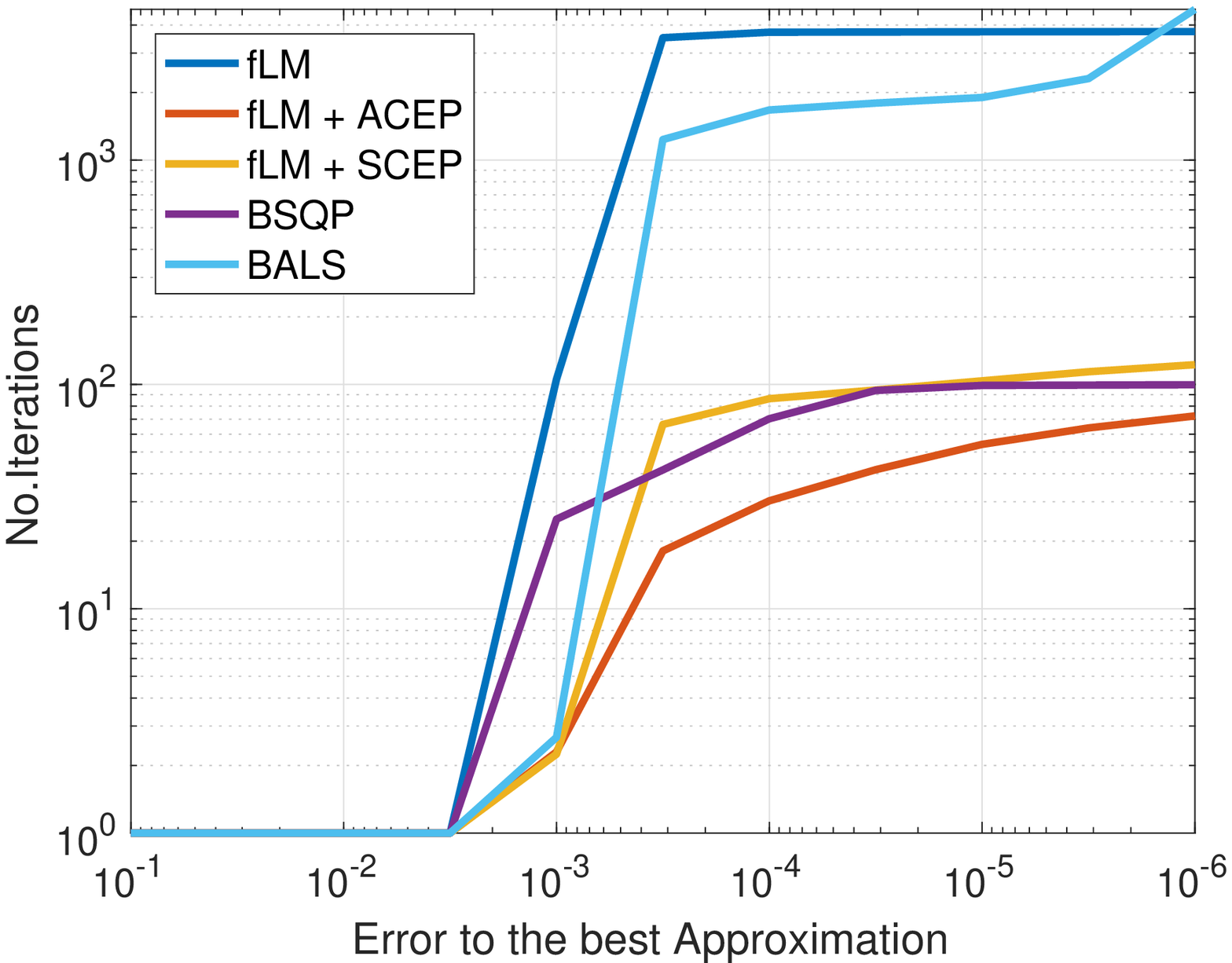}
\end{minipage}
\label{fig_ex_c4}}
\caption{Decomposition of the noise-free tensors in Example~\ref{ex4}. The first row compares success ratios of the considered algorithms. A relative error of $10^{-6}$ is considered perfect to attain for decomposition of a noise-free tensor. The second row compares the numbers of iterations of algorithms to attain the best relative error. The results were reported over 150 independent runs. 
}
\label{fig_bals3mc_2}
\end{figure}

For the same tensors, we applied algorithms for the bounded CPD. The bound of the norm of rank-1 tensors was adjusted during the estimation. The BALS seemed to achieve higher success ratios than fLM for tensors with $I_n = 7$ and $I_n = 12$. It could explain the tensors with a relative error of $10^{-5}$ in 60-70\% of the runs. The BSPQ achieved a much higher success ratio than the BALS.

In another assessment, we compare the number of iterations of algorithms to achieve the best relative error. For example, in order to achieve an error of $10^{-6}$ to the best relative error, the fLM algorithm might need a thousand of iterations, while this algorithm with ACEP and SCEP needed on average 72 and 122 iterations. BSQP required 400 iterations as shown in Fig.~\ref{fig_ex_a4} for decomposition of tensors of size $4 \times 4 \times 4$. This is because the algorithm iterated to adjust the bound of the norm of rank-1 tensors.


As seen in Fig.~\ref{fig_bals3mc_2}, when the algorithms reached the approximation error of $10^{-4}$, they quickly converged to the approximation error of $10^{-8}$. In total, the number of iterations of the three algorithms, fLM+ACEP, fLM+SCEP and BSQP, were at most comparable.
In summary, the BSQP, Interior Point method for bounded norm constrained CPD (BITP) and fLM with the EPC explained the noise-free tensors with a nearly perfect accuracy 
in less than 300 iterations.

\begin{figure}[t]
\centering
\subfigure[$I = 4$, $R = 5$]{
\begin{minipage}{.31\linewidth}
\includegraphics[width=1\linewidth, trim = 0.0cm -.4cm 0cm 0cm,clip=true]{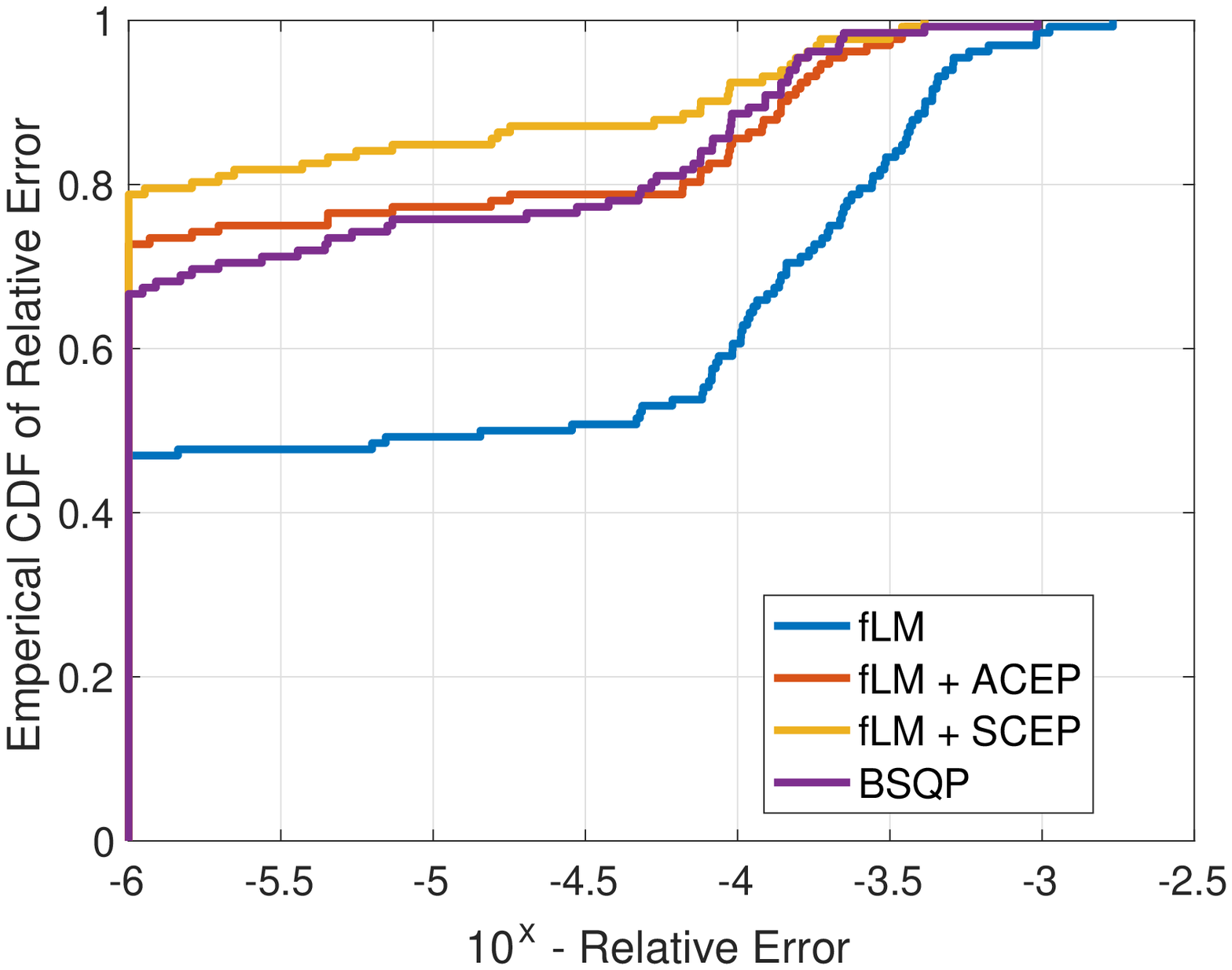}\\\includegraphics[width=1\linewidth, trim = 0.0cm -.4cm 0cm 0cm,clip=true]{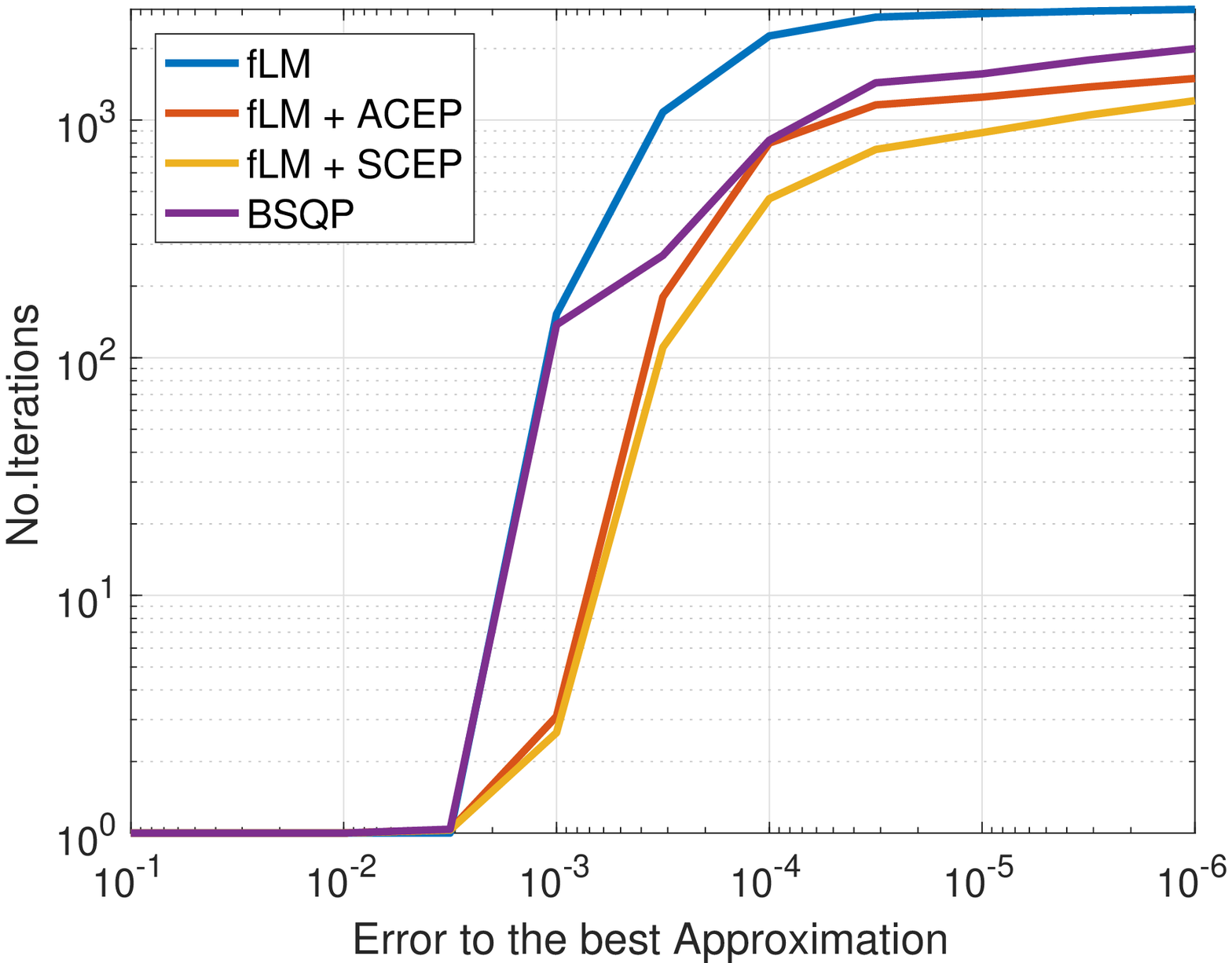}
\end{minipage}
\label{fig_ex_a5}
}
\subfigure[$I = 7$, $R = 10$]{
\begin{minipage}{.31\linewidth}
\includegraphics[width=1\linewidth, trim = 0.0cm -.4cm 0cm 0cm,clip=true]{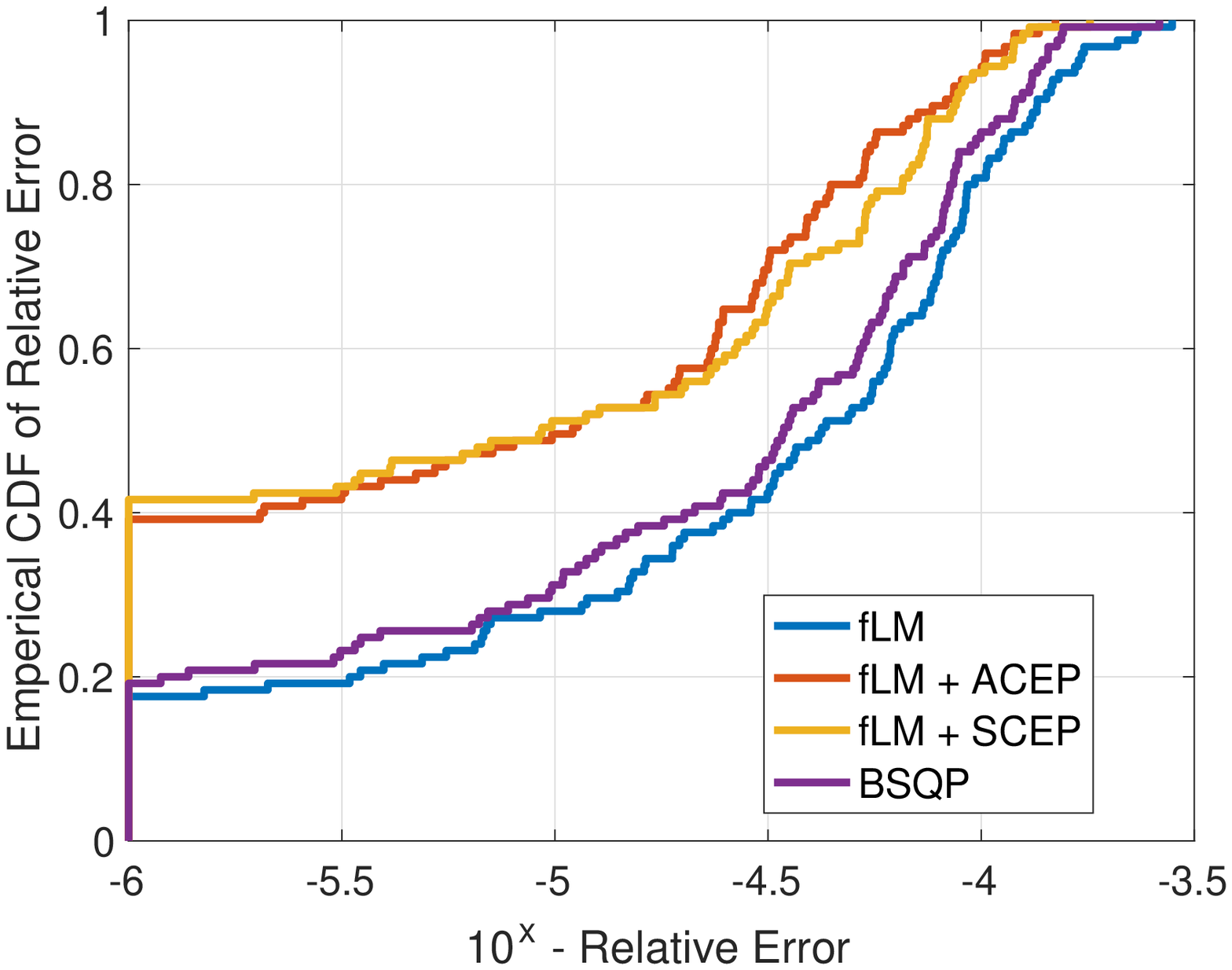}\\
\includegraphics[width=1\linewidth, trim = 0.0cm -.4cm 0cm 0cm,clip=true]{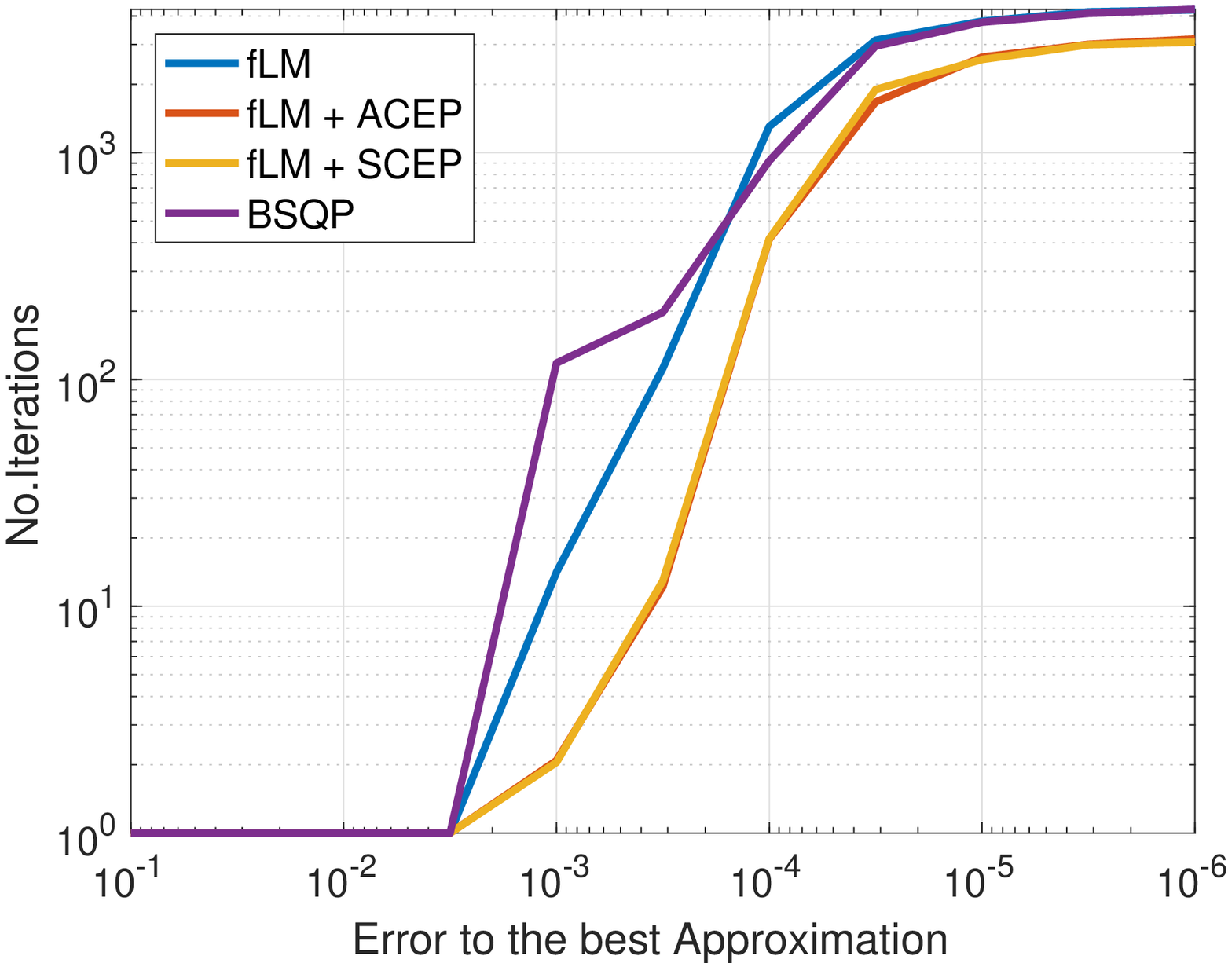}
\end{minipage}
\label{fig_ex_b5}
}
\caption{Decomposition of the noisy tensors in Example~\ref{ex4}. The first row compares success ratios of the considered algorithms. A relative error of $10^{-6}$ is considered perfect to attain for decomposition of a noise-free tensor. The second row compares the numbers of iterations of algorithms to attain the best relative error. The results were reported over 150 independent runs.}
\label{fig_bals3mc_4noisy}
\end{figure}

For the test cases with noisy tensors, we added Gaussian noise into the noise-free tensors at a signal-to-noise ratio of SNR = 50 dB. The success ratio and the number of iterations to achieve the best relative error are illustrated in Fig.~\ref{fig_bals3mc_4noisy}.
The fLM attained an approximation error of $10^{-6}$ to the best relative error in 47\% and 18\% of runs for the tensors with sizes of $I_n = 4$ and $I_n = 7$, respectively, while BSQP met the approximation error of $10^{-6}$ in 67\% and 20\% of run. The results confirm that the EPC method, either ACEP or SCEP, gained the success ratio of the fLM up to 79\% and 42\%, respectively, while the algorithm demanded a lower number of iterations than fLM. 

%

%

\end{example}

 \begin{example}[Decomposition of block tensors]\label{ex_bcd}
 
 This example was inspired by the block-term decomposition of 
the tensors which had rank exceeding the dimensions, and  highly collinear loading components. We constructed the tensors from two blocks of size $6 \times 6 \times 6$, each of rank 6, and collinearity degrees among the loading components were within a range of $[0.95, 0.999]$
 \be
 \tY = \tI \times_1 \bU_{1,1}  \times_2 \bU_{1,2}  \times_3 \bU_{1,3} +  \tI  \times_1 \bU_{2,1}  \times_2 \bU_{2,2}  \times_3 \bU_{2,3} \, ,
 \ee
 where $\tI$ represents the diagonal tensor.
Our experience is that the tensors are difficult to decompose 
for most of the conventional CP techniques.
We ran the fastALS algorithm in 10 iterations to generate initial values. 

The fLM did not complete the decomposition of the noise-free tensors within the error range of $10^{-6}$ even after 3000 iterations as seen in Fig.~\ref{fig_bcd2blocks}. The reason is that the norm of estimated rank-1 tensors was relatively large, on average around 3994.3. Using the EPC methods, e.g., ACEP or SCEP, we reduced the norm to 11.8. By this way, the fLM converged in a few hundreds of iterations as illustrated in Fig.~\ref{fig_res_mc_bcd_R12_cpd_noiter} for one run of the decomposition. In Fig.~\ref{fig_bcd2blocks}, ``ACEP+fLM'' stands for the combination of the EPC after 10 iterations of the fastALS, and the fLM, whereas  ``fLM+SCEP+fLM'' represents the process of three stages: running fLM until it stopped, then applying SCEP to correct the rank-1 tensors, and running the fLM again. 

The results confirm that the proposed correction method worked efficiently. When using with this, the fLM could complete the decomposition in more than 80\% of runs. The results were reported over 130 independent runs.

\begin{figure}[t]
\centering
\subfigure[]{\includegraphics[width=.48\linewidth, trim = 0.0cm 0cm 0cm 0cm,clip=true]{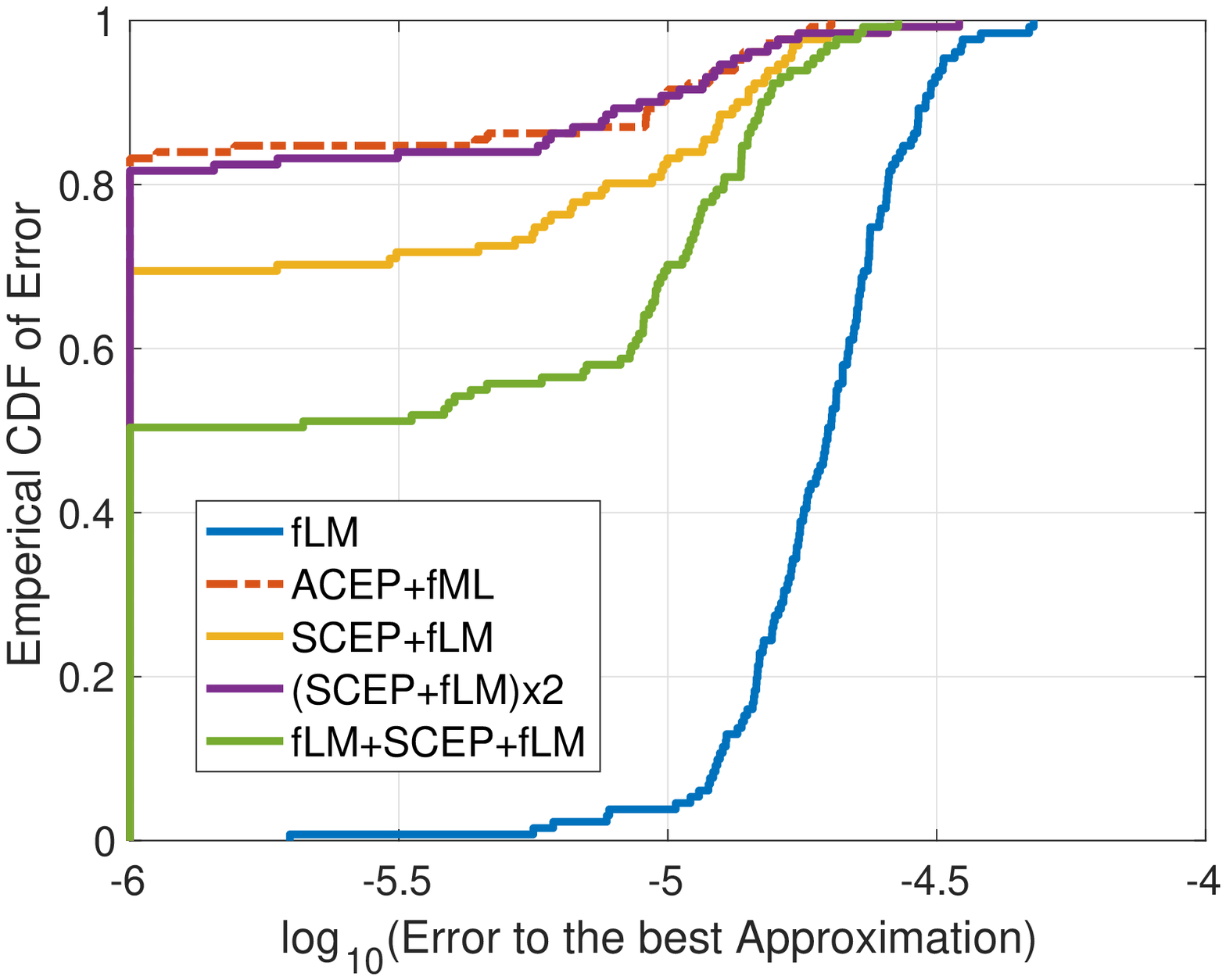}\label{fig_res_mc_tv_data_R12_cpd_error}}
\subfigure[]{\includegraphics[width=.48\linewidth, trim = 0.0cm 0cm 1.9cm 1cm,clip=true]{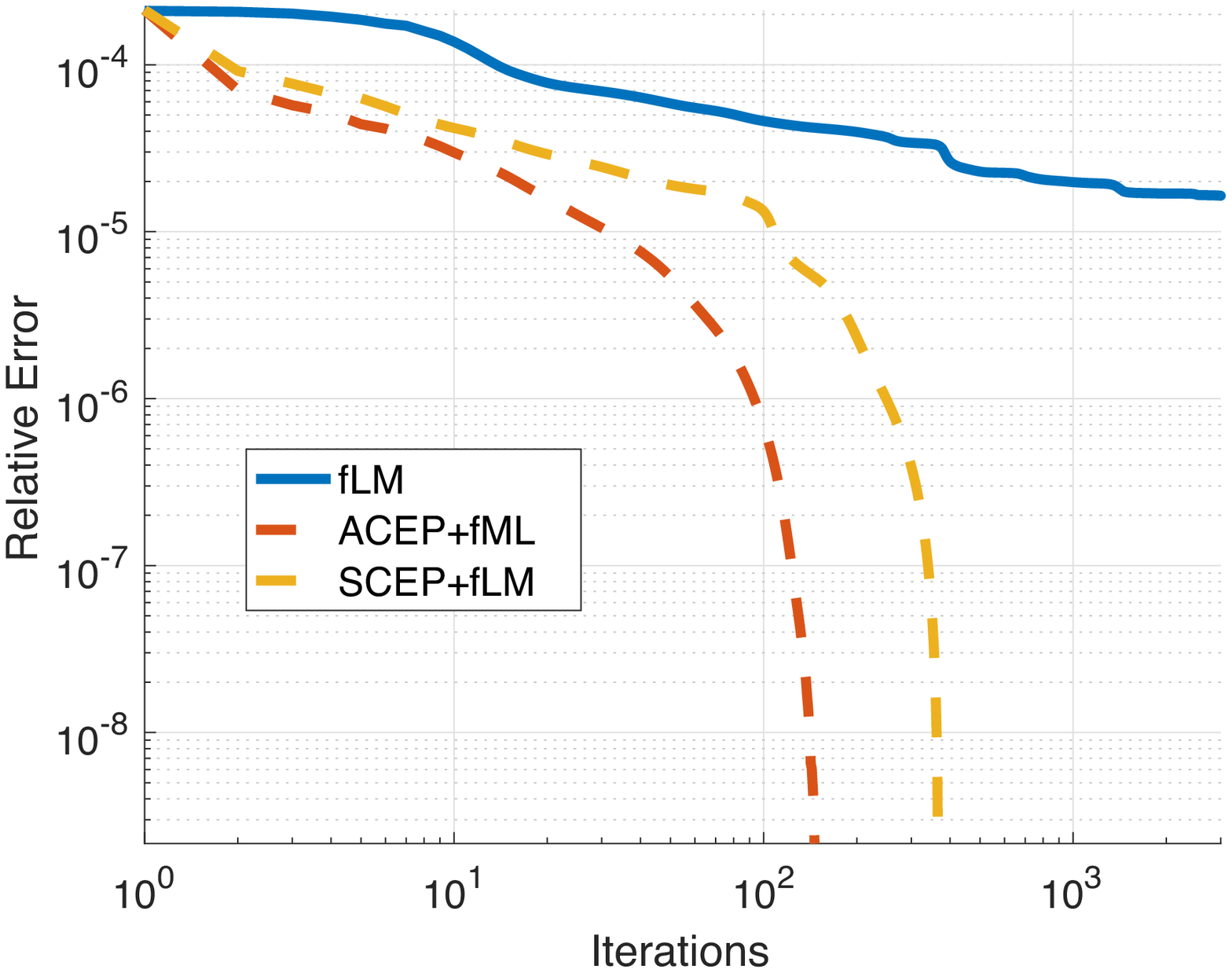}\label{fig_res_mc_bcd_R12_cpd_noiter}}
\caption{(a) Success ratios of the fLM algorithm with and without EPC in Example~\ref{ex_bcd}. (b) The relative error of the fLM algorithms in one run of the CP decomposition.
}
\label{fig_bcd2blocks}
\end{figure}

\end{example}

\begin{example}[Decomposition of tensor for multiplication of two matrices of size $3 \times 3$]\label{ex_3x3}

In this example, we compare the performance of algorithms for CPD with  and without EPC and CPD with a bounded rank-1 tensor norm. The decomposed tensor is the multiplication tensor for the case of two matrices of size $3 \times 3$. This tensor is of size $9 \times 9 \times 9$, contains only zeros and ones, and obeys 
\be
	\vtr{\bA \bB} = \tY \times_1 \vtr{\bA^T}^T \times_2 \vtr{\bB^T}^T \, \notag
\ee
where $\bA$ and $\bB$ are of size $3 \times 3$.
The tensor is considered of rank-$R=23$.
In \cite{journals/corr/TichavskyPC16}, we developed an LM algorithm to update the vector of parameters which is assumed to be on a ball with a prescribed diameter. 

Decomposition of this tensor using ALS or LM often gets stuck in false local minima, 
or requires a huge number of iterations. This is because the norm of rank-1 tensors is significantly large as seen in Fig.~\ref{fig_normvserror_mx3x3x} for the results using the fLM algorithm \cite{Phan_fLM}.

For this case, we used the estimated tensor obtained after 10 runs using the fLM algorithm to initialise the parameters in the BSPQ and BITP algorithm for the bounded CPD. 
The bound of the rank-1 tensor norm was set to $\epsilon = 15$.
The results show that the two algorithms converged after a few tens of iterations. This is much faster than using fLM without EPC. 

In another comparison, we corrected rank-1 tensors of the obtained results using the fLM algorithm. The new tensor after the EPC was then used as initial values for the BSQP and BITP  for the bounded CPD and the fLM algorithms \cite{Phan_fLM}.
The results are compared in Fig.~\ref{fig_cp_correction_}.
For this later test, the three algorithms converged after 10 iterations. 

\begin{figure}[t]
\centering
\subfigure[]{\includegraphics[width=.48\linewidth, trim = 0.0cm 0cm 0cm 0cm,clip=true]{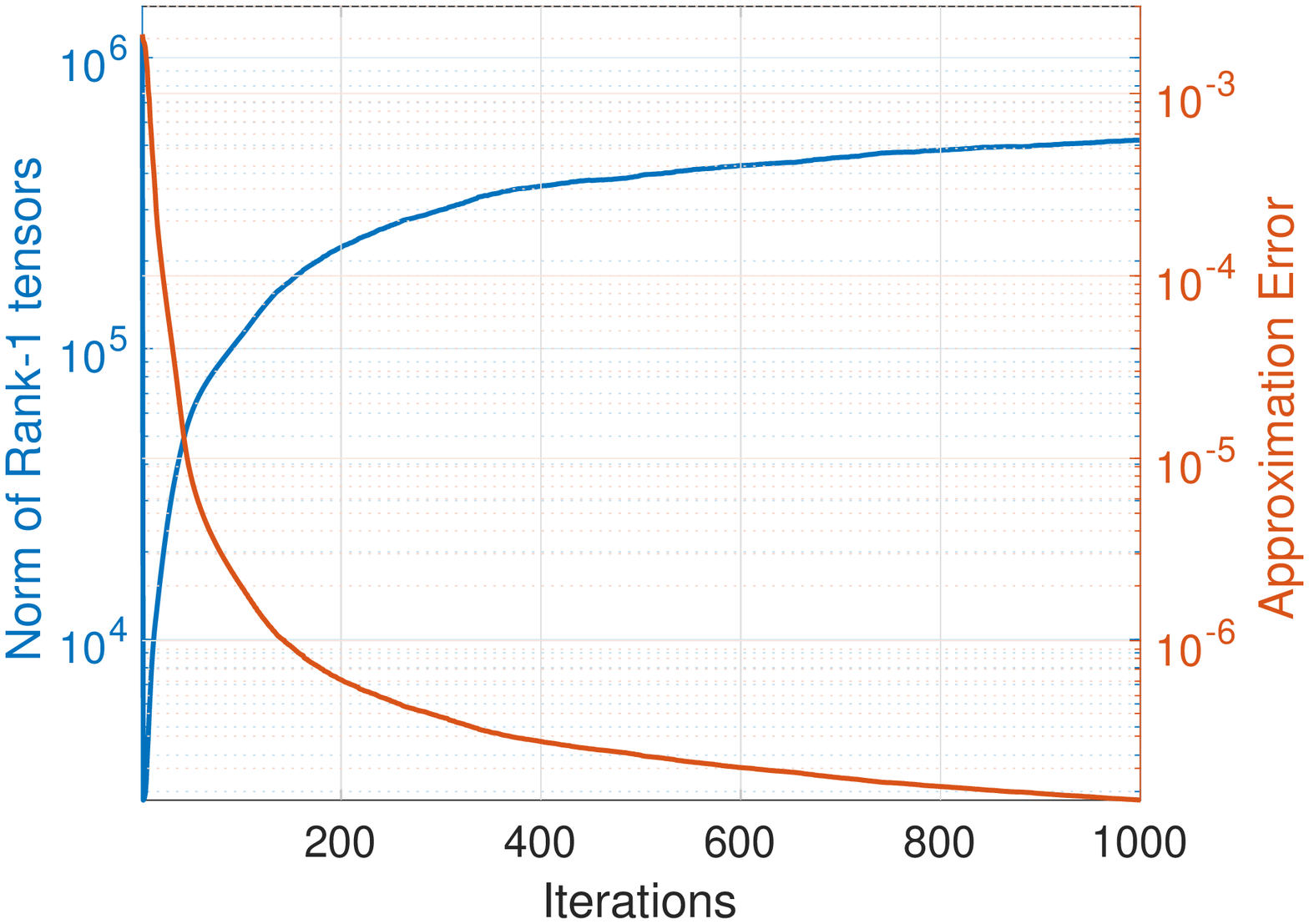}\label{fig_normvserror_mx3x3x}}\hfill
\subfigure[]{\includegraphics[width=.45\linewidth, trim = 0.0cm 0cm 0cm 0cm,clip=true]{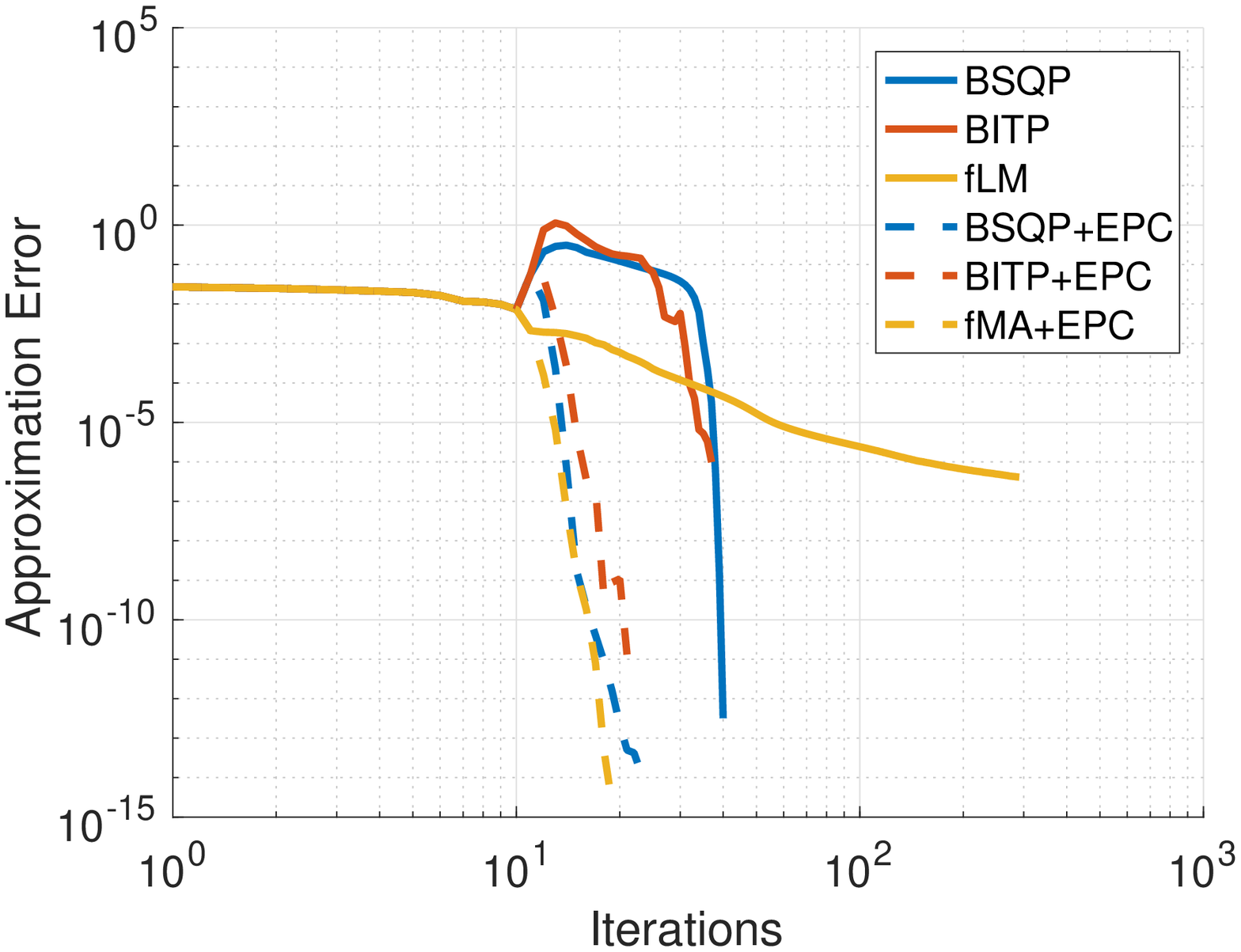}}
\caption{Comparison of the approximation errors in CPD of the multiplicative tensors of size $9 \times 9 \times 9$ which has rank of $R = 23$. The three considered algorithms are ran from the same initial point which is generated by executing the fLM algorithm in 10 iterations. 
After 10 iterations, the estimated tensor is corrected so that its rank-1 tensors have minimal tensor norms. The algorithms continue the decomposition with and without EPC. 
}
\label{fig_cp_correction_}
\end{figure}

\end{example}

%
%
%
%
%
%
%

\begin{example}[Decomposition of the TV-ratings data \cite{Lundy}]\label{ex_tvdata}

We decomposed the TV-ratings data \cite{Lundy} which comprises 16 rating scales $\times$ 15 American TV shows $\times$ 30 subjects. This data is well known to illustrate the degeneracy in CPD for example with the rank $R$ = 2, 3 or 4 \cite{Harshman04,journals/siammax/Stegeman12}. Here we compared the fLM algorithm with and without the EPC for the decomposition of rank-$R = 10$. 
We ran the ALS algorithm in 100 iterations to generate the initial parameters, then executed the fLM algorithm.
For the EPC method, the bound of the approximation error was set to 1.01 times of the approximation error of the initial point. 
 The success ratios of the considered algorithms are plotted in Fig.~\ref{fig_tv_err}.
In 74.6\% of runs, the relative errors obtained by the fLM were very close to the best results, with a difference less than $10^{-6}$. 
The success ratio of fLM was improved after executing the EPC either with ACEP or SCEP (see Fig.~\ref{fig_tv_err}). In Fig.~\ref{fig_tvdata}(b), we illustrate the relative errors of algorithms as a function of the number of iterations in one run. The fLM started from a lower error but got stuck in a false local minimum after 100 iterations. Since the error bound was set to higher than the approximation error, the fLM with EPC started from higher relative error, but in the final, this algorithm achieved a lower approximation error as seen in Fig.~\ref{res_mc_tvdata_R10_err_vs_iter_fixed}.

\begin{figure}[t]
\centering
\subfigure[]{\includegraphics[width=.48\linewidth, trim = 0.0cm 0cm 0cm 0cm,clip=true]{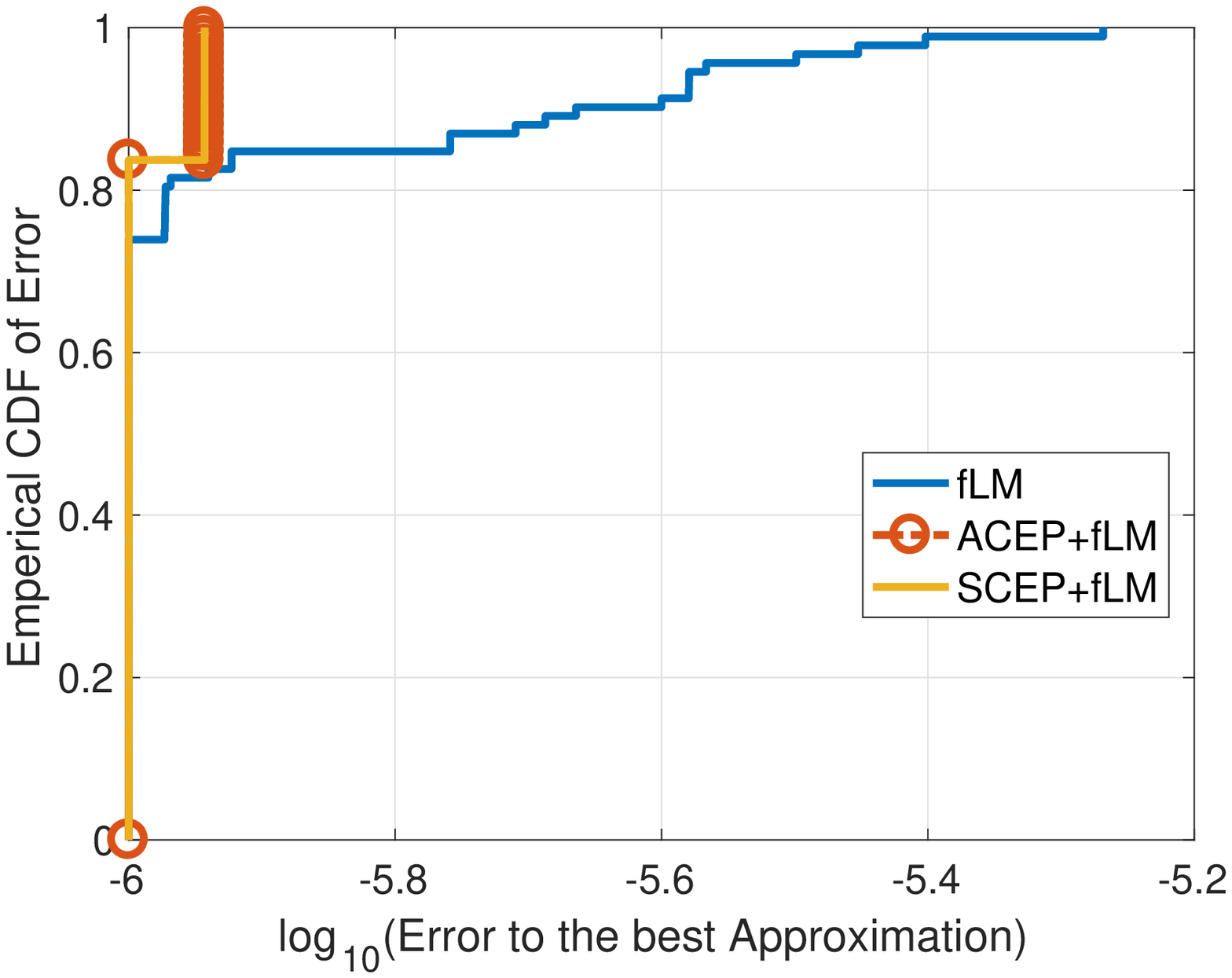}\label{fig_tv_err}}
\subfigure[]{\includegraphics[width=.48\linewidth, trim = 0.0cm 0cm 0cm 0cm,clip=true]{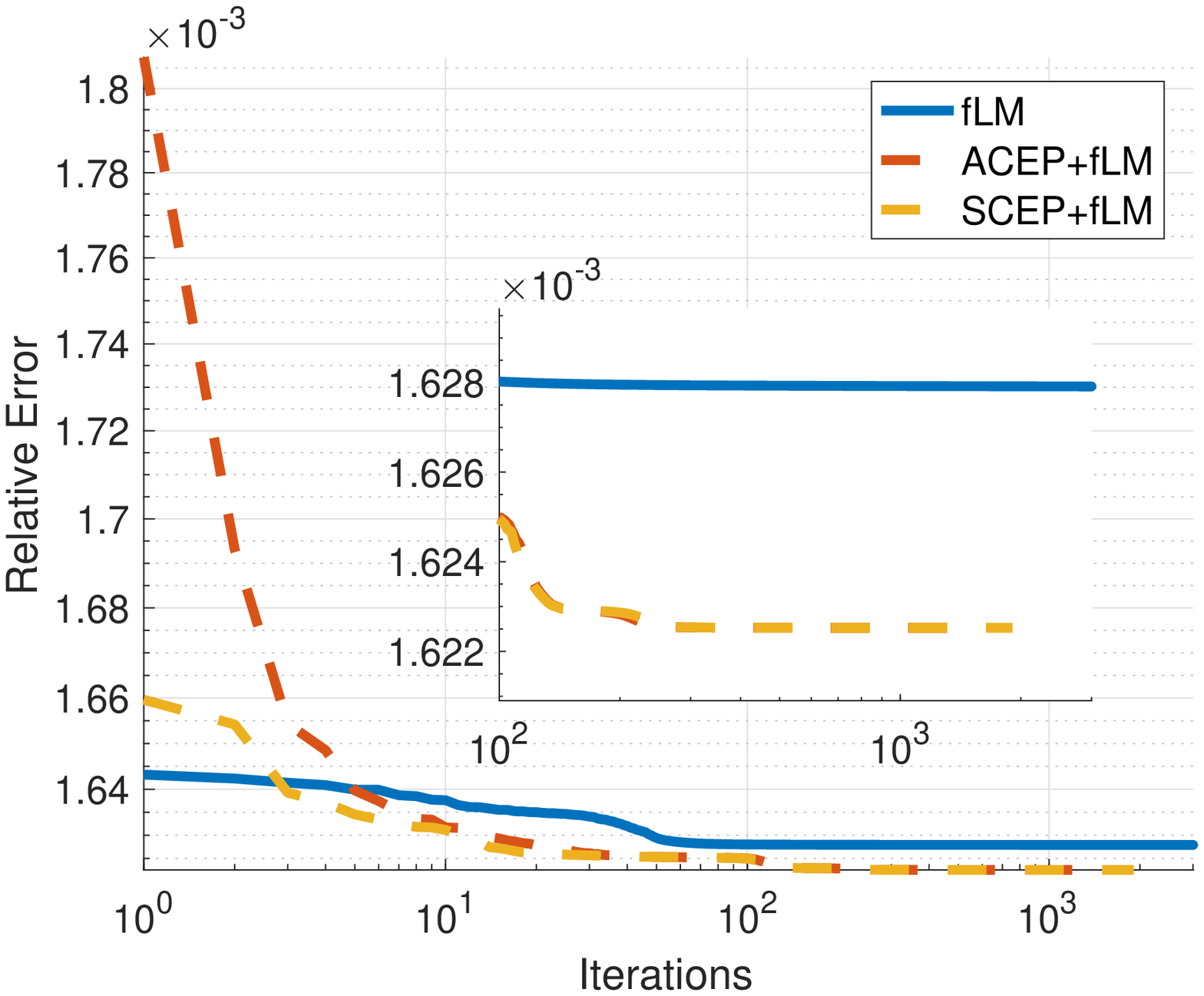}\label{res_mc_tvdata_R10_err_vs_iter_fixed}}
\caption{(a) Success ratios of the fLM algorithms with and without EPC in decomposition of the TV-ratings data. (b) Illustration of the changes of the relative errors in one run of the estimation. The fLM got stuck in a false local minima after at most 100 iterations.
}
\label{fig_tvdata}
\end{figure}

\end{example}

\section{Conclusions}\label{sec::conclusions}

In difficult scenarios of the CP tensor decomposition, when large loading components may cancel each other, we propose to seek
new decompositions with the same approximation error but with a minimum norm of the rank-one components.
In particular, we derive solutions to two constrained optimization problems, one for the error preserving correction method, and another one for the bounded CPD.
 The factor matrices in the two problems can be updated in closed-form in an alternating update scheme through solving Spherical Constrained Quadratic Programming. In addition, the SQP-based all-at-once algorithms have been developed to update all the parameters in the two problems at a time, but with a low complexity for the inversion of the Hessian matrices.
The relation between the new alternating algorithms with the ordinary ALS algorithm has been also presented.  
In simulations, we confirmed the efficiency of the proposed algorithms in the decomposition of tensors with rank exceeding the tensor dimensions (multiplication tensors) and on tensors with highly collinear rank-one components.

\appendices

\section{Linear Regression with a Bound Constraint}\label{sec::linreg_bounderror}

The linear regression problem with a constraint on the regression error is stated as below 
\be
\min_{\bx}  \quad & \|\bx\|^2   \quad 
\text{subject to} &   \|\by - \bA \bx \| \le  \delta  \label{eq_linreg_bound}
\ee 
where $\by$ is a vector of length $I$ of dependent variables, $\bA$ is a regressor matrix of size $I \times K$, and a nonnegative regression bound $\delta$. 

It is obvious that if $\delta \ge \|\by\|$, then the zero vector $\bx = \0$ is a minimiser to (\ref{eq_linreg_bound}).
Therefore, in order to achieve a meaningful regression, the regression bound $\delta$ needs to be in the following range. 

\begin{lemma}[Range of the bound $\delta$]
The problem (\ref{eq_linreg_bound}) has a minimiser of nonzero entries when 
\be
\|{\bPi}_{\bA}^{\perp}  \, \by \|  \le \delta <  \| \by\|
\ee
where  ${\bPi}_{\bA}^{\perp}$ is an orthogonal complement of the column space of $\bA$.
\end{lemma}
\begin{proof}
Let $\bU$ be an orthogonal basis for the column space of $\bA$. Then 
\be
	\delta^2 \ge \|\by - \bA \bx \|^2 = \|\bU^T \by - \bU^T \bA \bx\|_F^2 + \|{\bPi}_{\bA}^{\perp}  \, \by\|^2  \ge \|{\bPi}_{\bA}^{\perp}  \, \by\|^2 \,. \notag 
\ee
\end{proof}

For simplicity, we assume that $\bA$ is full rank matrix, otherwise, we solve the problem with a compressed regressor matrix with a smaller bound 
\be
\min\quad \|\bx\|^2 \quad \text{subject to} \quad \|\hat{\by} - \hat{\bA} \bx\|  \le  \hat{\delta}
\ee
where $\hat{\by} = \bU^T\by$, $\hat{\bA} = \bU^T \bA$, and $\hat{\delta}^2 = \delta^2 - \|{\bPi}_{\bA}^{\perp}  \, \by\|^2$.

We show that the inequality sign in (\ref{eq_linreg_bound}) can be replaced by the equality sign. 
\begin{lemma} \label{lem_linreg_bounderror}
The minimiser to (\ref{eq_linreg_bound}) is the minimiser to the following problem 
\be
\min_{\bx}  \quad & \|\bx\|^2 \quad 
\text{subject to} &   \|\by - \bA \bx \|  =  \delta   \label{eq_linreg_circle}
\ee
\end{lemma}

%

\begin{proof}
From the Lagrangian of the optimisation in  (\ref{eq_linreg_bound})
\be
\calL(\bx,\lambda) = \frac{1}{2} \|\bx\|^2  +  \frac{\lambda}{2}  (\|\by - \bA \bx\|^2 - \delta^2) \notag 
\ee
the stationarity condition indicates that $\lambda$ must be non-zero, otherwise $\bx = \0$
\be
\nabla_{\bx} \calL =  (\bI + \lambda \, \bA^T \bA) \bx - \lambda \bA^T \by   = \0 \notag .
\ee
From the complementary slackness condition
$
\lambda (\|\by - \bA \bx\|^2 - \delta^2) = 0 $
and since $\lambda >0$, the constraint must hold the equality, i.e., $\|\by - \bA \bx\|  =  \delta$.
\end{proof}
We present an algorithm when the matrix of regressors $\bA$ is of full column rank, $ K \le I$.

Let $\bA = \bU \diag(\bs) \bV^T$ be an SVD of $\bA$, where  $\bV$ is an orthonormal matrix of size $K \times K$, and $\bs = [s_1, \ldots, s_K]> 0$. Hence ${\bPi}_{\bA}^{\perp}  = \bI - \bU \, \bU^T$.

Let $\hat{\by} = \bU^T \by$,  $\hat{\delta} = \sqrt{\delta^2  -   \|{\bPi}_{\bA}^{\perp}  \, \by\|^2}$, 
$\bz = \displaystyle \frac{1}{\hat{\delta}} (\hat{\by} -  \diag(\bs) \bV^T \bx)$, then 
\be
\bx &=& \bV \diag(\bs^{-1})(\hat{\by}- \hat{\delta} \bz)  \\
\|\bx\|_F^2 &=&  (\hat{\by} - \hat{\delta} \bz)^T \diag(\bs^{-2})  (\hat{\by} -\hat{\delta}  \bz) \notag \\
\|\by - \bA \bx \|^2 &=&  \|\bU^T \by - \diag(\bs) \bV^T \bx\|_F^2 + \|{\bPi}_{\bA}^{\perp}  \, \by\|^2  = 
\hat{\delta}^2\, \|\bz\|^2 + \|{\bPi}_{\bA}^{\perp}  \, \by\|^2 \, .\notag 
\ee
By this reparameterization, the problem (\ref{eq_linreg_circle}) becomes a QP over a sphere which can be solved in closed-form, e.g., see \cite{GANDER1989815,Phan_QPS}
\be
\min_{\bz}  \quad &    \bz^T \diag(\hat{\delta} \bs^{-2}) \bz - 2  \, \hat{\by}^T \diag(\bs^{-2}) \bz    \label{eq_linreg_boundz}\\
\text{subject to} &   \bz^T \bz  = 1 \notag.
\ee

\section{A Simplification Method For SCQP with Identical Eigenvalues}\label{sec:scqp_identical_eig}

We consider a QP problem over a sphere 
\be
&\min  \quad  & \frac{1}{2} \, {\tilde\bx}^T \, \diag({\bs})  \, {\tilde\bx} + \bc^T {\tilde\bx}\label{equ_qp_sphere2} \\
&\text{subject to}  \quad & {\tilde\bx}^T {\tilde\bx} = 1 \notag 
\ee
where $\bc^T \bc = 1$, and $\bs = [s_1 = 1 \le s_2 \le \cdots \le s_K]$.

We denote $J$ the number of distinct eigenvalues, $\tilde{\bs} = [\tilde{s}_1 = 1 < \tilde{s}_2 < \cdots < \tilde{s}_{J}]$, over a set of $K$ eigenvalues, $s_k$,  in (\ref{equ_qp_sphere2}), and classify $\bc = [{\bc}_{1}, {\bc}_{2}, \ldots, {\bc}_{J}]$ into $J$ sub-vectors, and  each ${\bc}_{j}$ consists of entries $c_{k}$ such that $s_k = \tilde{s}_j$, i.e., $\bc_j = [c_{k \in \calI_j}]$, where $\calI_j = \{k:  s_k = \tilde{s}_j\}$. 
In addition, we define a vector 
\be
\tilde{\bc}  = [\|\bc_1\|, \|\bc_2\|, \ldots, \|\bc_J\|]\,.
\ee
Then the following relation holds.
\begin{lemma}\label{lem_sqp_simplify}
The minimiser to (\ref{equ_qp_sphere2}) can be deduced from the minimiser to the SCQP with distinct eigenvalues 
\be
&\min  \quad  & \frac{1}{2} \, \bz^T \, \diag(\tilde{\bs} )  \, {\bz} + \tilde\bc^T {\bz}   \quad 
\text{subject to}  \quad   {\bz}^T {\bz} = 1 \notag \,,
\ee
as follows
\begin{itemize}
\item For non zero  $\tilde{c}_j$, 
$
\bx_{\calI_j} = \frac{z_j}{\tilde{c}_j}  \, \bc_j$
\item If $c_1= 0$ and $ d^2 = \displaystyle \sum_{j = 2}^{J} \frac{\tilde{c}_j^2}{(\tilde{s}_j  - 1)^2} \le 1$, 
$\bx_{\calI_1}$ can be arbitrary vectors on the ball $\|\bx_{\calI_1}\|^2 = 1 - d^2$,
\item Otherwise for zeros $\tilde{c}_j$, $\bx_{\calI_j}$ all are zeros. 
\end{itemize}
\end{lemma}
\begin{proof}
We consider a simple case when some eigenvalues are identical, e.g., $s_1 = s_2 = \cdots = s_L < s_{L+1} < \ldots < s_K$.
If $\bc_{1:L}$ are all zeros, the objective function is independent of $\tilde{\bx}_{1:L} = [\tilde{x}_1, \tilde{x}_2 \ldots, \tilde{x}_L]$, hence, $\tilde{\bx}_{1:L}$ can be any point on the ball $\|\tilde{\bx}_{1:L}\|^2 = d^2 = 1 - \sum_{k = L+1}^{K} \tilde{x}_{k}^2$.
Otherwise, $\tilde{\bx}_{1:L}$ is a minimiser to the constrained linear programming while fixing the other parameters $\tilde{x}_{L+1}, \ldots, \tilde{x}_K$
\be
	\min \quad  \bc_{1:L}^T \, \tilde{\bx}_{1:L} \quad \text{subject to} \quad   \|\tilde{\bx}_{1:L}\| = d  
\ee
which yields 
\be
\tilde{\bx}_{1:L}= \frac{-d}{\|\bc_{1:L}\|}  \bc_{1:L}\,.
\ee
For both cases, we can define  
\be
\bz &=& [-d, \tilde{x}_{L+1}, \ldots, \tilde{x}_{K}] ,\notag \\
\tilde{\bc} &=& [\|\bc_{1:L}\|, c_{L+1}, \ldots, c_{K}], \notag \\
\tilde{\bs} &=& [s_1, s_{L+1}, \ldots, s_K], \notag 
\ee
and perform a reparameterization to estimate $\bz$ from a similar constrained QP but with distinct eigenvalues $\tilde{\bs}$
\be
&\min  \quad  & \frac{1}{2} \, \bz^T \, \diag(\tilde{\bs} )  \, {\bz} + \tilde\bc^T {\bz}\label{equ_qp_sphere3}  \quad 
\text{subject to}  \quad   {\bz}^T {\bz} = 1 \notag \,.
\ee
Similarly, we can convert (\ref{equ_qp_sphere2}) to a problem with  $\tilde{s}_1 < \tilde{s}_2 < \cdots < \tilde{s}_J$. 
Now based on the fact of SCQP that for zero coefficients $\tilde{c}_j$, $z_j^{\star}$ will also be zeros, except for only the case $c_1 = 0$ and $1 \ge d^2$ \cite{Phan_QPS}.
\end{proof}

\section{SCQP with Matrix-variates}\label{sec:sqp_matrixvariate}

We consider an SCQP for a matrix-variate $\bX$ of size $I\times R$ given in the form of 
\be
\min\quad f(\bX) = \frac{1}{2} \, \tr(\bX^T \bQ \bX)  + \tr(\bB^T \bX)  \quad \text{s.t.} \quad \|\bX\|_F^2 = 1 \,\label{eq_spq_matrix} ,
\ee
where $\bQ$ is a psd matrix of size $I \times I$  and $\bB$ is of size $I \times R$.
The objective function can be rewritten in a similar form to  (\ref{equ_qp_sphere2}) as
\be
f(\bX) 
&=&  \frac{1}{2}  {\bx}^T (\diag(\bsigma) \otimes \bI_R)    {\bx}  +  \bv^T  {\bx} \notag 
\ee
where 
${ \bx} = \vtr{\bX^T \bU}$, $\bv = \vtr{\bB^T \bU}$ and $\bQ = \bU \diag(\bsigma)\bU^T$ is an EVD of $\bQ$. Due to the Kronecker product, each eigenvalue $\sigma_i$, $i = 1, \ldots, I$, is replicated $R$ times. 
Let $\bz^{\star}$ of length $I$ be a (unique) minimiser to an SCQP 
\be
\min\quad \frac{1}{2} \, \bz^T  \diag(\bsigma) \bz  +  \bc^T \bz  \quad \text{s.t.} \quad \bz^T \bz = 1 \,\label{eq_spq_matrix_2} ,
\ee
where  $\bc = [c_1, \ldots, c_I]$,  $c_i =  \|\bB^T \bu_i\|$. 
According to Lemma~\ref{lem_sqp_simplify}, 
for a nonzero coefficient $c_i$, $\bx_i =  \frac{z_i}{c_i} \bB^T \bu_i$, otherwise, $\bx_i$ can be any vector on the ball $\bx_i^T \bx_i =  z_i^2$ for a zero vector $\bB^T \bu_i$.
%


\section{Gradient and Hessian of the Objective Function $f(\btheta)$ in (\ref{eq_cp_boundnorm2b})}\label{sec:gradHess_ftheta}

Let $\bbeta_n =[\bu^{(n) T}_{1} \bu^{(n)}_{1}, \ldots, \bu^{(n)T}_{R} \bu^{(n)}_{R}]$
and $\bbeta_{-n} =  \bigcircledast_{k \neq n} \, \bbeta_{n = 1}^{N}$, $\bbeta = \bigcircledast_{n} \, \bbeta_{n}$.
The gradient $\bg_{f}$ and Hessian $\bH_{f}$ of the objective function w.r.t. to $\bU^{(n)}$ are given by 
\be
\bg_{f} &=& \left[\ldots, \vtr{\frac{\partial f}{\partial \bU^{(n)}}}^T,  \ldots  \right]^T \notag \\
&=&   \left[\ldots,  \vtr{\bU^{(n)}  \diag(\bbeta_{-n})}^T ,\ldots \right]^T \\
\bH_f &=& \nabla^{2} f =  \bD + 2 \, \bV\, \bF \, \bV^T
\ee
where $\bF = [\bF_{n,m}]$ is an $N \times N$ partitioned matrix of matrices $\bF_{n,m}$ with $\bF_{n,n} = \0$ and $\bF_{n\neq m} = \diag(\bbeta_{-(n,m)})$, and 
\be
\bD &=& \diag([\bbeta_{-1} \otimes \1_{I_1}, \ldots, \bbeta_{-N} \otimes \1_{I_N}])  \, ,\\
\bV &=& \blkdiag(
\bV_1, \ldots, \bV_N), \quad \bV_n = \blkdiag(\bu^{(n)}_{1}, \ldots, \bu^{(n)}_{R})  . 
\ee
The Hessian $\bH_f$ can also be given in an equivalent form of a block diagonal matrix and a rank-$R$ adjustment 
\be
\bH_f =  \blkdiag(\ldots, \diag(\bbeta_{-n} \otimes \1_{I_n}) - 2 \, \tilde{\bV}_n \, \diag(\bbeta) \, \tilde{\bV}_n^T, \ldots )  + 2\, \tilde{\bV} \, \diag(\bbeta)\, \tilde{\bV}^T \, , \label{eq_Hessian_obj}
\ee
where $\tilde{\bV}_n = \bV_n  \diag(\1\oslash \bbeta_{n})$ and $\tilde{\bV} = [\tilde{\bV}_1^T, \ldots, \tilde{\bV}_N^T]^T$  is of size $R(\sum_{n} I_n) \times R$.

\section{Gradient and Hessian of the Constraint Function $c(\btheta)$ in (\ref{eq_cp_boundnorm2b})}\label{sec::gradHessian_c_theta} 

According to Theorem~2\cite{Phan_fLM},  
the gradient and Hessian of the constraint function $c(\btheta)$ w.r.t $\btheta$ are given by 
\be
\bg_c &=&  \left[\ldots,  \vtr{\bU^{(n)}  \bGamma_{-n} - \bY_{(n)} \left(\bigodot_{k\neq n}  \bU^{(n)}\right)}^T ,\ldots \right]^T  \\
\bH_c &=& \bG + \bZ \bK \bZ^T 
\ee
where  
\be
\bG &=& \blkdiag(\bGamma_{-n} \otimes \bI_{I_n}) \notag\,, \\
\bZ &=& \blkdiag(\ldots, \bI_{R} \otimes \bU^{(n)}, \ldots) \notag \,, \\
\bK &=& [\bK_{n,m}] , \quad \bK_{n,n} = \0, \quad \bK_{n\neq m} = \dvec({\bGamma_{-(n,m)}}) \,\notag .
\ee 
The Hessian $\bH_c$ can also be expressed as  \cite{PetrfLMnr6}
\be
\bH_c = \blkdiag(\bGamma_{-n} \otimes \bI_{I_n} - \tilde{\bZ}_n \, \bPsi  \, \tilde{\bZ}_n^T) + \tilde{\bZ} \, \bPsi \, \tilde{\bZ}^T \, \label{eq_Hessian_cp}
\ee
where $\tilde{\bZ} = [\tilde{\bZ}_n]$, $\tilde{\bZ}_n = (\bI_{R} \otimes \bU^{(n)}) \dvec(\1\oslash {\bGamma_n})$, $\bPsi = \bP_{R,R} \dvec({\bGamma})$.
Note that 
$\tilde{\bV} = \tilde{\bZ}(:,[1, R+1, \ldots, R^2])$ and $\bbeta = \diag(\bGamma)$, $\bbeta_{-n} = \diag(\bGamma_{-n})$.

\section{Proof of the Identity in (\ref{eq_Prr_identity})}\label{sec_Prr_identity}
Denote by $\overline{i,j} = R(j-1)+i$ the linear index of $(i,j)$, the identity matrix of size $R^2\times R^2$ can be represented as $\bI_{R^2} = \left[\ve_{\overline{1,1}},\ldots,   \ve_{\overline{i,j}}, \ldots,\ve_{\overline{R,R}} \right] $.
From the definition $\vtr{\bX_{R\times R}} = \bP_{R,R} \vtr{\bX_{R\times R}^T}$, it is obvious that 
\be
\bP_{R,R} = \bP_{R,R} \, \bI_{R^2} =  \left[\ldots, \bP_{R,R} \, \ve_{\overline{i,j}}, \ldots \right]  = \left[\ldots,   \ve_{\overline{j,i}}, \ldots \right]  \, .  \notag 
\ee
For an arbitrary matrix $\bA$ of size $R \times R$, we can express the diagonal matrix of $\vtr{\bA}$ as 
\be
\dvec(\bA) &=& \left[ \ldots, a_{i,j} \, \ve_{\overline{i,j}}, \ldots \right] 
=   \left[ \ldots,  a_{i,j}   \bP_{R,R} \, \ve_{\overline{j,i}} , \ldots \right] 
\notag\\
&=&
\left[ \ldots, \bP_{R,R} \, \dvec(\bA^T) \,  \ve_{\overline{j,i}} , \ldots \right] = \bP_{R,R} \, \dvec(\bA^T) \,   \left[\ldots,   \ve_{\overline{j,i}}, \ldots \right]   \notag \\
&=&  \bP_{R,R} \, \dvec(\bA^T) \, \bP_{R,R} \,. \notag.
\ee
Note that $\bP_{R,R}$ is a symmetric matrix, hence, (\ref{eq_Prr_identity}) is obtained straightforwardly. 
 
\bibliographystyle{IEEEbib}
\bibliography{bibligraphy_thesis,BIBTENSORS2016,BIBTENSORS2017}

\end{document}